\documentclass[10pt,twoside]{article}

\usepackage{natbib}
\usepackage[margin=1in]{geometry}
\usepackage{amsmath}
\usepackage{amsthm}              
  {
      \newtheorem{assumption}{Assumption}
      \newtheorem{remark}{Remark}
      \newtheorem{theorem}{Theorem}
      \newtheorem{lemma}{Lemma}
      \newtheorem{corollary}{Corollary}
      \newtheorem{proposition}{Proposition}
      \newtheorem{definition}{Definition}

  }
\usepackage{amsfonts}
\usepackage{amssymb}
\usepackage{hyperref}[6.83]
\hypersetup{colorlinks,
			linkcolor=[rgb]{.61,0,0.3},
			citecolor=[rgb]{.14,.47,.14}}

\usepackage[capitalize,nameinlink]{cleveref}[0.19]
\usepackage{booktabs}
\usepackage{longtable}
\usepackage{graphicx}
\usepackage{subcaption}
\usepackage{fancyhdr}
\usepackage[usenames,dvipsnames]{color}
\usepackage[ruled]{algorithm2e}
\usepackage{tikz}

\usepackage{csvsimple}
\usepackage{siunitx}
\usepackage{array}

\crefname{definition}{Definition}{Definitions}
\crefname{assumption}{Assumption}{Assumptions}
\crefname{theorem}{Theorem}{Theorems}
\crefname{remark}{Remark}{Remarks}
\crefname{lemma}{Lemma}{Lemmas}
\crefname{corollary}{Corollary}{Corollaries}
\crefname{proposition}{Proposition}{Propositions}
\crefname{section}{Section}{Sections}
\crefname{subsection}{Subsection}{Subsections}
\crefname{example}{Example}{Examples}
\crefname{table}{Table}{Tables}
\crefname{problem}{Problem}{Problems}
\crefname{algocf}{Algorithm}{Algorithms}
\crefname{figure}{Figure}{Figures}

\crefformat{equation}{\textup{#2(#1)#3}}
\crefrangeformat{equation}{\textup{#3(#1)#4--#5(#2)#6}}
\crefmultiformat{equation}{\textup{#2(#1)#3}}{ and \textup{#2(#1)#3}}
{, \textup{#2(#1)#3}}{, and \textup{#2(#1)#3}}
\crefrangemultiformat{equation}{\textup{#3(#1)#4--#5(#2)#6}}%
{ and \textup{#3(#1)#4--#5(#2)#6}}{, \textup{#3(#1)#4--#5(#2)#6}}{, and \textup{#3(#1)#4--#5(#2)#6}}

\Crefformat{equation}{#2Equation~\textup{(#1)}#3}
\Crefrangeformat{equation}{Equations~\textup{#3(#1)#4--#5(#2)#6}}
\Crefmultiformat{equation}{Equations~\textup{#2(#1)#3}}{ and \textup{#2(#1)#3}}
{, \textup{#2(#1)#3}}{, and \textup{#2(#1)#3}}
\Crefrangemultiformat{equation}{Equations~\textup{#3(#1)#4--#5(#2)#6}}%
{ and \textup{#3(#1)#4--#5(#2)#6}}{, \textup{#3(#1)#4--#5(#2)#6}}{, and \textup{#3(#1)#4--#5(#2)#6}}

\crefdefaultlabelformat{#2\textup{#1}#3}

\newcommand{\norm}[1]{\left\Vert #1 \right\Vert}
\newcommand{\1}[1]{\textbf{1}\left[ #1 \right]}
\newcommand{\Prb}[1]{\mathbb{P}\left[ #1 \right]}
\newcommand{\E}[1]{\mathbb{E}\left[ #1 \right]}
\newcommand{\tr}[1]{\mathbf{tr}\left[  #1 \right]}
\newcommand{\bigO}[1]{\mathcal{O}\left[ #1 \right]}

\newcommand{\condPrb}[2]{\mathbb{P}\left[\left. #1 \right\vert #2 \right]}

\newcommand{\linspan}[1]{\mathrm{span}\left[ #1 \right]}

\newcommand{\range}[1]{\mathrm{range}\left( #1 \right)}
\newcommand{\nullsp}[1]{\mathrm{ker}\left( #1 \right)}
\def \PP {\mathbb{P}}

\DeclareMathOperator{\row}{row}
\DeclareMathOperator{\rrow}{rrow}

\title{An Implicit Representation and Iterative Solution of Randomly Sketched Linear Systems}
\author{Vivak Patel, Mohammad Jahangoshahi \& Daniel Adrian Maldonado}
\date{}

\begin{document}

\maketitle 

\begin{abstract}
Randomized linear system solvers have become popular as they have the potential to reduce floating point complexity while still achieving desirable convergence rates. One particularly promising class of methods, random sketching solvers, has achieved the best known computational complexity bounds in theory, but is blunted by two practical considerations: there is no clear way of choosing the size of the sketching matrix \textit{apriori}; and there is a nontrivial storage cost of the sketched system. In this work, we make progress towards addressing these issues by implicitly generating the sketched system and solving it simultaneously through an iterative procedure. As a result, we replace the question of the size of the sketching matrix with determining appropriate stopping criteria; we also avoid the costs of explicitly representing the sketched linear system; and our implicit representation also solves the system at the same time, which controls the per-iteration computational costs.

Additionally, our approach allows us to generate a connection between random sketching methods and randomized iterative solvers (e.g., randomized Kaczmarz method, randomized Gauss-Seidel). As a consequence, we exploit this connection to (1) produce a stronger, more precise convergence theory for such randomized iterative solvers under arbitrary sampling schemes (i.i.d., adaptive, permutation, dependent, etc.), and (2) improve the rates of convergence of randomized iterative solvers at the expense of a user-determined increases in per-iteration computational and storage costs. We demonstrate these concepts on numerical examples on forty-nine distinct linear systems.
\end{abstract}

\section{Introduction}
Over the past few decades, randomized linear system solvers have become popular as they have the potential to reduce floating point complexity or maintain limited memory footprints, while still achieving desirable convergence rates \citep[e.g.,][]{strohmer2009,woodruff2014}. In particular, the noniterative class of randomized linear system solvers, based on random matrix sketching \citep[see][]{woodruff2014}, have exceptionally low computational complexities, at least in theory. 
Unfortunately, the theoretical promise of these random matrix sketching solvers is blunted by their practical limitations: there is no clear way of choosing the size of the sketching matrix and there is a nontrivial storage cost of the projected system \citep{mahoney2016}. In fact, the practical challenges of random matrix sketching solvers have prevented them from being fully embraced by the numerical optimization community \citep[e.g.,][]{nocedal2018}.

In this work, we begin to address these two primary practical issues of random matrix sketching, which we recall are: the challenge of choosing the size of the sketching matrix, and the challenge of storing the projected system. Our main insight is to recast the separate sketch-\textit{then}-solve core of random sketching methods into an equivalent, iterative sketch-\textit{and}-solve, in which the sketching matrix is generated incrementally without being explicitly stored and the system is incrementally solved from the implicitly derived sketched matrix.\footnote{It is worth mentioning that the random sketch solvers have been used iteratively in a different sense \citep[e.g., see][]{gower2015}: the noniterative scheme is simply repeated in order to get better convergence properties. We are not doing this, but rather turning the noniterative scheme into an iterative one.} As a result of our approach, (1) we can implicitly grow the size of the sketching matrix until a user-determined stopping criteria is reached without having to determine the size of the sketching matrix apriori; (2) we implicitly represent the sketched system without having to explicitly store the projected system, which allows us to avoid the cost of storing the projected system;
and (3) we can naturally implement random sketching solvers within distributed and parallel computing paradigms. Thus, our approach of converting the usual sketch-\textit{then}-solve procedure to a sketch-\textit{and}-solve procedure begins to address the aforementioned practical challenges of random matrix sketching.

Moreover, our approach provides a bridge between the newer concerns around sketching-based solvers and more classical areas of applied mathematics research such as stopping criteria. One such bridge is the placement of random sketching methods and (what we will call) base randomized iterative methods\footnote{We will be more precise about what we refer to as base methods. For now, such methods are exemplified by randomized Kaczmarz \citep{strohmer2009} and randomized Gauss-Seidel \citep{leventhal2010}.} on a single spectrum of procedures, which has several immediate consequences.

First, the number of rows of the sketching matrix that results in the solution (this number is a random quantity) connects to an alternative rate-of-convergence result for general base randomized iterative methods that guarantees a rate-of-convergence less than one for arbitrary sampling schemes---even for underdetermined systems (\cref{theorem:no-mem-convergence}). Consequently, our results complement and improve on previous results in several ways. In particular, we allow for arbitrary sampling schemes, not just sampling schemes that are independent and identically distributed as in \cite{gower2015} (Lemma 4.2), \cite{richtarik2017} (Theorem 4.8), \cite{zouzias2013} (Theorem 3.4), and \cite{ma2015} (Equation 3.10). Moreover, our results do away with the exactness assumption \citep[see][Assumption 2]{richtarik2017}, and precisely characterize the inexactness that can occur for arbitrary sampling schemes (\cref{theorem:no-mem-convergence,theorem:no-mem-column-convergence}). 
Additionally, our results define convergence on a maximal set---effectively, a set occurring with probability one for sampling schemes of interest---, which builds on the work of \cite{chen2012}. As example applications of our results, we supply rates of convergence with probability one for random permutation sampling methods (\cref{theorem-no-mem-w-o-replacement}) and independent, identically distributed sampling schemes (asymptotically, see \cref{theorem-no-mem-w-replacement}). 
As a more interesting application of our results, we specify generic conditions for the convergence of a broad class of adaptive schemes (see \cref{subsection:adaptive-sampling}), which can account for the maximum residual scheme, the maximum distance scheme, schemes that randomize over a greedy subset, and schemes that are greedy over randomized subsets \citep{motzkin1954,gubin1967,lent1976,censor1981,nutini2016,bai2018,haddock2019}. 
We note that the rates that we provide as examples are rather loose in comparison to results that are specialized to each case, yet our results often supply information that is \textit{not} available in these other results as discussed above.

Second, we can generate a series of ``intermediate'' procedures between sketching methods and base methods that trade-off between computational resources (e.g., floating-point operations, storage) and rates of convergence. Thus, we can take a sketching method and reduce its computational footprint in exchange for a slower rate of convergence, or increase the computational footprint of base methods to improve their rate of convergence (\cref{alg: rank-one RPM low mem}). Moreover, these ``intermediate'' procedures can be readily parallelized as we discuss in \cref{section: overview}.

Finally, by shifting our perspective from improving the sketch-\textit{then}-solve procedure to improving the performance of base methods, we find that our approach is a randomized orthogonalization procedure in the row space of the coefficient matrix of the linear system. Thus, by presenting our approach from this latter perspective, we will simplify the introduction and the related theory of our approach. Now, before pursuing this further, we reiterate our main contributions.
\begin{enumerate}
\item First, we turn the typical sketch-\textit{then}-solve noniterative random sketching solver into an iterative, sketch-\textit{and}-solve method, which lays a foundation for addressing the previously enumerated practical challenges of random sketching solvers: there is no clear way of choosing the size of the sketching matrix \textit{apriori}; and there is a nontrivial storage cost of the sketched system.
\item Second, through our approach, we place random sketching methods and base randomized iterative methods (e.g., randomized Kaczmarz, randomized Gauss-Seidel, and Sketch-and-Project \citep{gower2015}) on a single spectrum of methods.
\item Third, owing to this connection, we are able to generate ``intermediate'' methods between random sketching and base methods, which can trade-off between computational resources and rates of convergence. 
\item Fourth, owing to this connection, we use the geometric implications of random sketching methods to develop an alternative rate-of-convergence result for general base methods for arbitrarily determined systems and \textit{arbitrary sampling schemes}, which advances the with-probability-one results of \cite{chen2012}, generalizes the deterministic cyclic results in \cite{bai2013,galantai2005,wallace2014}, complements the mean-squared error results of \cite{richtarik2017}, and accounts for a litany of adaptive methods considered in \cite{motzkin1954,gubin1967,lent1976,censor1981,nutini2016,bai2018,haddock2019}.
\item Finally, we provide a generic set of conditions for characterizing a broad class of adaptive methods, and, from these conditions, prove convergence and rate-of-convergence results for a number of classical and emerging adaptive methods in the literature under a unified framework (see \cref{subsection:adaptive-sampling}).
\end{enumerate}

The remainder of this paper is organized as follows. In \cref{section: overview}, we introduce our procedure; we state the connection between our procedure and random sketching methods, which allows us to convert the less practical sketch-\textit{then}-solve approach to our sketch-\textit{and}-solve approach; and, finally, we introduce our general algorithm and variants for low-memory environments, shared memory environments, distributed memory environments, and large, sparse, structured linear systems. In \cref{section:full-memory,section:no-memory}, we develop the convergence theory for the two methodological extremes---sketching and base methods---leaving the intermediate, more complex cases to future work, and discuss particular examples. In \cref{section: experiments}, we test our algorithms on forty-nine distinct linear systems. In \cref{section:conclusion}, we conclude this work and preview future efforts.

\section{Our Procedure} \label{section: overview}
While our motivating application is to address the practicality of random sketching methods, our approach is best introduced from the perspective of base randomized iterative methods. Here, we review the basic formulation of randomized iterative methods (\cref{subsection:overview}), which we then use to heuristically introduce our general procedure (\cref{subsection:derivation}). We then refine our procedure for the case of rank-one methods, such as Randomized Kaczmarz and Randomized Gauss-Seidel, which allows us to restate random sketching from a sketch-\textit{then}-solve procedure to a sketch-\textit{and}-solve procedure (\cref{subsection:rank-one-and-sketching}). We conclude this section with comments on algorithmic refinements for parallel platforms (\cref{subsubsection:parallel}), limited memory platforms (\cref{subsection:low-memory}), and, for structured linear systems, limited communication platforms (\cref{subsubsection:structured}).

\subsection{A Brief Overview} \label{subsection:overview}
Let $A \in \mathbb{R}^{n \times d}$ and $b \in \mathbb{R}^n$ be the coefficient matrix and constant vector, respectively. Assuming consistency, our goal is to determine an $x^* \in \mathbb{R}^d$, not necessarily unique, such that
\begin{equation} \label{eqn: linear system}
Ax^* = b.
\end{equation}
In a base randomized iterative approach, a sequence of iterates $\lbrace x_k : k +1 \in \mathbb{N} \rbrace$ is generated that has the form
\begin{equation} \label{eqn: general random update}
x_{k+1} = x_k + V_{k}(b - Ax_k),
\end{equation}
where $V_k \in \mathbb{R}^{d \times n}$ are independent random variables, which we call residual projection matrices (RPM). The RPM defines the base technique which is being used. To make this formulation concrete, we give several examples of randomized iterative methods that have this formulation.

\paragraph{Randomized Kaczmarz.} Let $A_{i,} \in \mathbb{R}^d$ denote the $i^{\text{th}}$ row of $A$ and let $e_{i}$ denote the $i^{\text{th}}$ standard basis vector of dimension $n$. Define the random variable $I$ such that
$$ \Prb{I = i} = \begin{cases}
\frac{\norm{A_{i,}}_2^2}{\norm{A}_F^2} & i=1,\ldots,n \\
0 & \text{otherwise}
\end{cases}.$$
Now, given an independent copy of $I$ at each $k$, define the RPM, $V_{k} = A_{I,} e_{I}'/\norm{A_{I,}}_2^2.$ Then, using \cref{eqn: general random update},
\begin{align*}
x_{k+1} = x_k + A_{I,} e_{I}'(b - Ax_k) = x_k + A_{I,}(b_I - A_{I,}'x_k)/\norm{A_{I,}}_2^2,
\end{align*}
which is the Randomized Kaczmarz method of \citet{strohmer2009}. $\quad\blacksquare$

\paragraph{Randomized Gauss-Seidel.} Let $A_{,j} \in \mathbb{R}^n$ denote the $j^{\text{th}}$ column of $A$ and let $f_j$ denote the $j^{\text{th}}$ standard basis vector of dimension $d$. Define a random variable $J$ such that
$$ \Prb{J = j} = \begin{cases}
\frac{\norm{A_{,j}}_2^2}{\norm{A}_F^2} & j = 1,\ldots,d \\
0 & \text{otherwise}
\end{cases}.$$
Now, given an independent copy of $J$ at each $k$, define the RPM, $V_{k} = e_{J} A_{,J}'/\norm{A_{,J}}_2^2.$ Then, using \cref{eqn: general random update},
\begin{align*}
x_{k+1} = x_k + e_J A_{,J}'(b - Ax_k)/\norm{A_{,J}}_2^2,
\end{align*}
which is the Randomized Gauss-Seidel method of \citet{leventhal2010}. $\quad\blacksquare$

\paragraph{Randomized Block Coordinate Descent.} Let $t$ be a subset of $\lbrace 1,\ldots,d \rbrace$. Let $E_{t} \in \mathbb{R}^{d \times |\tau|}$ whose columns are the $d$-dimensional standard basis vectors whose non-zero components correspond to the indices in $t$. Let $\mathcal{T}$ be a partition $\lbrace 1,\ldots, d \rbrace$, and define a random variable $T$ that randomly selects a partition in $\mathcal{T}$. Given an independent copy of $T$ at each $k$, define the RPM, $V_k = ( E_T' A' A E_T )^\dagger E_T'A'$. Then, using \cref{eqn: general random update},
\begin{align*}
x_{k+1} = x_k + ( E_T' A' A E_T )^\dagger E_T'A'(b - A x_k),
\end{align*}
which is a version of the randomized block coordinate descent method specified by \citep[Equation 3.14]{gower2015}. $\quad\blacksquare$

\paragraph{Sketch-and-Project.} Let $\lbrace N_0,N_1,\ldots \rbrace$ be a sequence of sketching matrices with $n$ columns.  Define the $k^{\mathrm{th}}$ RPM to be $V_k = A'N_k'( N_k A A' N_k')^\dagger N_k$. Then, using \cref{eqn: general random update},
\begin{align*}
x_{k+1} = x_k + A'N_k'( N_k A A' N_k')^\dagger N_k (b - Ax_k),
\end{align*}
which is the general sketch-and-project method \citep[Equation 2.2]{gower2015}. $\quad\blacksquare$

\subsection{A Heuristic Derivation} \label{subsection:derivation}

Here, given a strategy for defining $\lbrace V_k : k +1 \in \mathbb{N} \rbrace$, we consider how to augment the randomized iterative method with prior information in order to improve convergence. For this purpose, we propose defining a sequence of matrices $\lbrace M_k : k +1 \in \mathbb{N} \rbrace \subset \mathbb{R}^{d \times d}$ (discussed below) and modify \cref{eqn: general random update} to be
\begin{equation} \label{eqn: optimal general random update}
x_{k+1} = x_k + M_k V_k (b - Ax_k).
\end{equation}
Of course, $M_k$ can simply be absorbed by $V_k$; however, our goal is to augment a randomized iterative method. For this reason, we will keep these two quantities separate. 

The main question now is how to choose $\lbrace M_k : k + 1 \in \mathbb{N} \rbrace$. Our guiding principle is that $M_k$ should minimize some measure of error between $x_{k+1}$ and $x^*$. However, implementing this guiding principle requires (1) choosing an appropriate error measure and (2) handling the fact that $x^*$ is unknown. In order to convey the intuition behind our procedure, we now state the heuristics that we use to make these choices.

\paragraph{Choosing an Error Measure.} Temporarily, suppose $x^*$ is known, and suppose we choose the $l^1$ error as our measure. Then, we must minimize the difference between the next iterate and $x^*$. While this error metric might have merit, solving it is a convex optimization problem that is as difficult to solve as the original linear system. Therefore, we will need an error measure which gives an explicit representation for $M_k$. Hence, one sensible choice is to use the Mahalanobis norm,
\begin{equation} \label{eqn: optimization problem}
\norm{x_{k+1} - x^*}_B^2,
\end{equation}
where $B$ is a positive definite, symmetric $\mathbb{R}^{d \times d}$ matrix.

\paragraph{Compensating for the Unknown Solution.} Now, we consider the task of compensating for the unknown $x^*$. 
For a fixed $x^*$ and for all $k+1 \in \mathbb{N}$, let $S_k = (x_k - x^*)(x_k - x^*)'$. Then, $S_{k+1}$ is related to $S_k$ by
\begin{align}
S_{k+1} &= (I - M_k V_k A) S_k (I - M_k V_k A)', \label{eqn: variance update}
\end{align}
where we have made use of \cref{eqn: optimal general random update}. 
Using \cref{eqn: variance update}, we can rewrite \cref{eqn: optimization problem} as
\begin{align*}
\norm{x_{k+1} - x^*}_B^2 = \tr{B (I - M_k V_k A) S_k (I - M_k V_k A)'}.
\end{align*}
To find an optimal $M_k$, we differentiate the right hand side and set the quantity equal to zero, which, explicitly is
\begin{equation} \label{eqn: linear matrix equation}
M_k (V_k A S_k A' V_k') - S_k A'V_k' = 0.
\end{equation}
Clearly, $V_kA {S}_k A'V_k'$ is positive semi-definite, so the solution to such a system will be the minimizer of the original objective function. However, \cref{eqn: linear matrix equation} may have many possible solutions or may fail to be consistent. In the case of nonunique solutions, we arbitrarily choose the solution with the smallest Frobenius norm. In the case of an inconsistent system, we arbitrarily choose the solution that minimizes the Frobenius norm of the residual and has the minimal Frobenius norm.  In both cases, a straightforward calculation gives
\begin{equation} \label{eqn: gain matrix}
{M}_k = S_k A' V_k' (V_k A S_k A' V_k')^\dagger,
\end{equation}
where $\dagger$ represents the Moore-Penrose Pseudo-inverse. Using \cref{eqn: gain matrix} with \cref{eqn: variance update}, we have the following recursion
\begin{equation} \label{eqn: updated variance update}
S_{k+1} = S_k - S_k A' V_k'(V_k A S_kA'V_k')^\dagger V_k A_k S_k.
\end{equation}
From \cref{eqn: gain matrix} and \cref{eqn: updated variance update}, it is clear that if $S_0$ were known, then the remaining unknown quantities could be determined. 

\paragraph{Our Procedure.} Since $S_0$ is unknown, we use the following heuristic procedure instead. First, we let $S_0 = I_d$, where $I_d$ is the $d$-dimensional identity matrix. Then, we recursively define $M_k$ and $S_k$ according to \cref{eqn: gain matrix,eqn: updated variance update}. To summarize, given $\lbrace V_k : k+1 \in \mathbb{N} \rbrace$, we let $S_0 = I_d$, let $x_0 \in \mathbb{R}^d$, and define
\begin{equation} \label{eqn: iterate update heur}
x_{k+1} = x_k + M_k V_k (b - Ax_k),
\end{equation}
where
\begin{equation} \label{eqn: gain matrix heur}
M_k = S_k A' V_k' (V_k A S_k A' V_k')^\dagger;
\end{equation}
and
\begin{equation} \label{eqn: variance update heur}
S_{k+1} = S_k - S_k A' V_k' (V_k A S_k A' V_k')^\dagger V_k A_k S_k.
\end{equation}

To interpret the terms in the above procedure, we begin by ignoring $S_k$ (i.e., set it to the identity). In this case, $M_k$ and its role in updating $x_k$ to $x_{k+1}$ is familiar: $M_k$ serves to map the residual onto the row space of $V_kA$, thereby ensuring that $x_{k+1}$ satisfies $V_k A x_{k+1} = V_k b$. If we now consider the role of $S_k$, we see that it is an orthogonal projector that ``weights'' the behavior of $M_k$ to ensure that $x_{k+1}$ satisfy $V_i A x_{k+1} = V_i b$ for $i \leq k$. We will see these interpretations clearly and formally when we focus on the case of rank-one $V_k$ next.

We pause here momentarily to discuss the relationship between our procedure, as specified by \cref{eqn: iterate update heur,eqn: gain matrix heur,eqn: variance update heur}, and the sketch-and-project method in \cite{gower2015} and \cite{richtarik2017}. At first glance, it may seem that our procedure is a special case of sketch-and-project with adaptive choices of the inner product at each iteration of the sketch-and-project update. Unfortunately, an effort to recast our approach as a special case of sketch-and-project breaks down at two fundamental points. First, the adaptive choices of the sketch-and-project inner product would have to be the inverse of $S_k$, which are orthogonal projection matrices. As a result, the inverse is ill-defined and the inner product is ill-defined. Of course, this can be rectified by allowing for a pseudo-metric, but this then results in the second major point of difficulty: the theory presented in \cite{gower2015} and \cite{richtarik2017} relies on the determinism and invertibility of the matrix defining the metric space to prove convergence. Thus, sketch-and-project, without a substantial investment, cannot readily include our approach. On the other hand, we can state sketch-and-project as a base randomized iterative approach, as shown in \cref{subsection:derivation}, and then improve on it with our procedure via \cref{eqn: iterate update heur,eqn: gain matrix heur,eqn: variance update heur}.

\subsection{Rank-One Refinements and Random Sketching} \label{subsection:rank-one-and-sketching}

By choosing $x_0 \in \mathbb{R}^d$ and $S_0 = I_d$, \cref{eqn: iterate update heur,eqn: gain matrix heur,eqn: variance update heur} describe an orthogonal projection procedure for typical randomized iterative procedures. However, because our goal is to improve the practicality of random sketching methods, we will need to focus on a particular refinement of the general procedure that occurs when $\lbrace V_k \rbrace$ are rank-one matrices, that is, when there exist pairs of vectors $\lbrace (v_k,w_k) \rbrace$ such that $V_k = v_k w_k'$ for each $k$. In this case, \cref{eqn: gain matrix heur,eqn: variance update heur} become
\begin{equation} \label{eqn: gain rank-one update}
M_k = \begin{cases}
\frac{1}{w_k'A S_k A'w_k \norm{v_k}_2^2} S_k A' w_k v_k' & S_k A'w_k \neq 0 \\
0 & \text{ otherwise,}
\end{cases}
\end{equation}
and
\begin{equation} \label{eqn: rank-one update-matrix}
S_{k+1}=  \begin{cases}
S_k - \frac{1}{w_k'A S_k A'w_k} S_k A'w_k w_k'AS_k  & S_k A'w_k \neq 0 \\
S_k & \text{ otherwise.}
\end{cases}
\end{equation}
Moreover, if we substitute \cref{eqn: gain rank-one update} into \cref{eqn: iterate update heur}, we recover

\begin{equation} \label{eqn: rank-one update-param}
x_{k+1} = \begin{cases}
x_k + \frac{1}{w_k'A S_k A'w_k} S_k A' w_k w_k'(b - Ax_k) & S_k A'w_k \neq 0 \\
x_k & \mathrm{otherwise}.
\end{cases}  \\
\end{equation}

It follows from \cref{eqn: rank-one update-param,eqn: rank-one update-matrix} that in the case of a rank-one RPM, \textit{the left singular vector of the RPM is not important}. To give some explicit examples, recall that rank-one RPM methods include the important special cases of randomized Kaczmarz and Gauss-Seidel. 

\paragraph{Randomized Kaczmarz with Orthogonalization.} Let $A_{i,} \in \mathbb{R}^d$ denote the $i^{\text{th}}$ row of $A$ and let $e_{i}$ denote the $i^{\text{th}}$ standard basis vector of dimension $n$. Define the random variable $I$ arbitrarily taking values in $\lbrace 1,\ldots,n \rbrace$. Now, given an independent copy of $I$ at each $k$, the randomized Kaczmarz method has rank-one RPM, $V_k = A_{I,} e_{I}'/\norm{A_{I,}}_2^2.$ Then, using \cref{eqn: rank-one update-param,eqn: rank-one update-matrix}, the randomized Kaczmarz method with orthogonalization is
$$
\begin{aligned}
x_{k+1} &= x_{k} + \frac{1}{e_I'A S_k A'e_I } S_k A' e_I e_I' (b - A x_k) \\
S_{k+1} &= \left( I_d - \frac{1}{e_I'A S_k A' e_I} S_k A' e_I e_I'A \right) S_k,
\end{aligned}
$$
when $S_k A' e_I \neq 0$, or is $x_{k+1} = x_{k}$ and $S_{k+1} = S_k$ otherwise. $\quad\blacksquare$

\paragraph{Randomized Gauss-Seidel with Orthogonalization.} Let $A_{,j} \in \mathbb{R}^n$ denote the $j^{\text{th}}$ column of $A$ and let $f_j$ denote the $j^{\text{th}}$ standard basis vector of dimension $d$. Define a random variable $J$ arbitrarily taking values in $\lbrace 1,\ldots,d \rbrace$. Now, given an independent copy of $J$ at each $k$, the randomized Gauss-Seidel method has rank-one RPM, $V_k = e_{J} A_{,J}'/\norm{A_{,J}}_2^2.$ Then, using \cref{eqn: rank-one update-param,eqn: rank-one update-matrix}, the randomized Gauss-Seidel method with orthogonalization is
$$
\begin{aligned}
x_{k+1} &= x_k + \frac{1}{A_{,J}'A S_k A'A_{,J}} S_k A' A_{,J} A_{,J}'(b - Ax_k) \\
S_{k+1} &= \left( I_d - \frac{1}{A_{,J}'A S_k A'A_{,J}} S_k A' A_{,J} A_{,J}'A \right) S_k,
\end{aligned}
$$
when $S_k A' A_{,J} \neq 0$, or is $x_{k+1} = x_{k}$ and $S_{k+1} = S_k$ otherwise. $\quad\blacksquare$

Again, we see from the two preceding examples that the left singular vector of the rank-one RPM does not play a role in the updates for our procedure. As we now explain, this observation is critical for converting the impractical, noniterative randomized sketch-\textit{then}-solve methods into iterative randomized sketch-\textit{and}-solve methods. 

Recall that the fundamental sketch-then-solve procedure is to construct a specialized matrix $N^\mathrm{sketch} \in \mathbb{R}^{k \times n}$, then generate and solve the smaller, sketched problem $(N^\mathrm{sketch}A)x = N^\mathrm{sketch}b$ \citep[see][Ch. 1]{woodruff2014}.\footnote{We note that the typical formulation considers linear regression rather than a linear system.} The special matrix $N^\mathrm{sketch}$, called the sketching matrix, can be generated in a variety of ways such as making each entry an independent, identically distributed Gaussian random variable \citep{indyk1998}, or by setting the columns of $N^\mathrm{sketch}$ as uniformly sampled columns (with replacement) of the appropriately-dimensioned identity matrix \citep{cormode2005}.

In order to convert the usual sketch-\textit{then}-solve procedure into our sketch-\textit{and}-solve procedure, we simply set $\lbrace w_k : k +1 \in \mathbb{N} \rbrace \subset \mathbb{R}^n$ 
to the transposed rows of $N^\mathrm{sketch}$, which we will rigorously demonstrate in \cref{section:full-memory}. Of course, this requires that we have a streaming procedure for generating arbitrarily many rows of $N^\mathrm{sketch}$. For concreteness, we show how to do this for the two sketching strategies just mentioned.

\paragraph{Random Gaussian Sketch.} In the random Gaussian sketch, the entries of the sketching matrix, $N^\mathrm{sketch}$, are independent, standard normal random variables. Accordingly, we let $\lbrace w_k \rbrace$ be independent, $n$-dimensional standard normal vectors. We see that if $N^\mathrm{sketch}$ has $r$ rows, then $N^\mathrm{sketch}$ and
\begin{equation*}
\begin{bmatrix}
w_0' \\
w_1' \\
\vdots \\
w_{r-1}'
\end{bmatrix}
\end{equation*}
have the same distribution. $\quad\blacksquare$

\paragraph{Count Sketch.} Fix $K \in \mathbb{N}$, and let $\lbrace E_1,E_2,\ldots \rbrace$ be drawn from the $\mathbb{R}^K$ standard basis vectors with replacement. Define a sequence of Rademacher random variables $\lbrace R_1,R_2,\ldots \rbrace$ which are independent and independent of $\lbrace E_1,E_2,\ldots \rbrace$. The count sketch sketching matrix, $N^\mathrm{sketch}$, is specified by
\begin{equation*}
\begin{bmatrix}
R_1 E_1 & R_2E_2 & \cdots & R_n E_n
\end{bmatrix},
\end{equation*}
which is a matrix whose entries are either $-1$, $0$ or $1$. Generally, the choice of $K$ is the topic of substantial theory and consideration \citep{cormode2005,clarkson2017}. Owing to the fact that we have a streaming procedure, we do not need to worry too much about $K$. Therefore, we generate $\lbrace w_k \rbrace$ as follows:
\begin{enumerate}
\item Generate a count sketch matrix with $K$ small. In our experiments below, we used $K = 10$. 
\item To generate a $w_k$, pop a row of the matrix and set it to $w_k$.
\item Once the count sketch matrix is exhausted, regenerate a new count sketch matrix with the same $K$. Repeat.
\end{enumerate}
From this strategy, (a) if we let $\lbrace N_{(i)}: i \in \mathbb{N} \rbrace$ denote a sequence of independent $K \times n$ count sketch matrices, (b) $i_k$ denote the remainder of an integer $k$ divided by $K$ and incremented by one, and (c) we let $\lbrace e_i \rbrace$ denote the standard basis vectors of $\mathbb{R}^K$, then $w_k = N_{(i_k)}'e_{i_k}$ for all $k+1 \in \mathbb{N}$. $\quad\blacksquare$

Thus, if we let \texttt{RPMStrategy()} define a generic user-defined procedure for choosing $\lbrace w_k : k +1 \in \mathbb{N} \rbrace$, then this observation gives us \cref{alg: rank-one RPM full mem} for (1) converting the sketch-\textit{then}-solve procedure into a sketch-\textit{and}-solve procedure, and (2) adding orthogonalization to such base methods as randomized Kaczmarz and randomized Gauss-Seidel.  

\begin{algorithm}[ht]
\KwData{Initialization $x_0$, \texttt{RPMStrategy()} for $w_0,w_1,\ldots$, \texttt{TerminationCriteria()}}
\KwResult{Estimate $\hat{x}$}
\BlankLine
$k \leftarrow 0$ \\
$S \leftarrow I_d$ \medskip

\While{ TerminationCriteria() == false}{
\tcp{Compute search direction}
$w_k \leftarrow RPMStrategy()$ \\
$q_k \leftarrow A'w_k$ \\
$u_k \leftarrow S_kq_k$ \medskip 

\tcp{Check if $S_k A'w_k = 0$}
\If{$u_k == 0$}{ $k \leftarrow k+1$ \\ continue to next iteration } \medskip

\tcp{Update Iterate}
$r_k \leftarrow b'w_k - q_k'x_k$ \\
$\gamma_k \leftarrow u_k'q_k$ \\
$x_{k+1} \leftarrow x_k + u_k \left(r_k/\gamma_k \right)$ \medskip

\tcp{Update Projection Matrix}
$S_{k+1}  \leftarrow (I - \frac{1}{\gamma_k} u_k q_k' )S_k$ \medskip 

\tcp{Update Iteration Counter}
$ k \leftarrow k+1$

}
\Return{$x_{k+1}$}
\caption{Rank-One RPM Method \label{alg: rank-one RPM full mem}}
\end{algorithm}

\subsection{Algorithmic Refinements Considering the Computing Platform}

\cref{alg: rank-one RPM full mem} implicitly assumes the traditional sequential programming paradigm. However, the performance of the algorithm can be improved by taking advantage of parallel computing architectures.  
Here, we will consider a handful of important computing architecture abstractions and how our procedure can adapt to different configurations. In \cref{subsubsection:parallel}, we will consider the case of a parallel computing architecture for which the communication overhead, which is proportional to the dimension $d$, is not a limiting factor. For this subsection, the problems that we have in mind come from data and imaging sciences, where $n \gg d$ and $d$ is reasonably sized. In \cref{subsection:low-memory}, we consider a similar class of problems where the communication of $\bigO{d}$-sized vectors is acceptable and $n \gg d$, but that $d$ is so large that storing and manipulating a matrix in $\mathbb{R}^{d \times d}$ is burdensome. Finally, in \cref{subsubsection:structured} we will consider problems in which computational overhead becomes a bottleneck for scalability, but that we have structured systems that will allow us to circumvent this issue. For this ultimate subsection, the problems that we have in mind here come from the solution of systems of differential equations \citep[e.g.,][]{dongarra1986}.

\subsubsection{Asynchronous Parallelization on Shared and Distributed Memory Platforms} \label{subsubsection:parallel}
First, when we are using a matrix sketch for \texttt{RPMStrategy()}, one of the expensive components of the computation is determining $\begin{bmatrix} A & b \end{bmatrix}'w_k$. Fortunately, in our sketch-and-solve procedure, this expensive computation can be trivially asynchronously parallelized on a shared memory platform when
\begin{enumerate}
\item the data within the rows $\begin{bmatrix} A & b \end{bmatrix}$ are stored together, and
\item the \texttt{RPMStrategy()} generates $\lbrace w_k : k+1 \in \mathbb{N} \rbrace$ that are either independent (e.g., the Gaussian Strategy) or can be grouped into independent subsets (e.g., the Count-Sketch strategy).
\end{enumerate} 
When these two requirements are met, each processor can generate its own $\lbrace w_k : k+1 \in \mathbb{N} \rbrace$ independently of the other processors, and evaluate $\begin{bmatrix} A & b \end{bmatrix}'w_k$. It can then simply write the resulting row to an address reserved for performing the iterate and $S_k$ matrix updates by the master processor. Importantly, this procedure does not require locking any of the rows of $\begin{bmatrix} A & b \end{bmatrix}$, and the reserved addresses can use fine grained locks to prevent any wasted calculations.

Similarly, in our sketch-and-solve procedure, computing $\begin{bmatrix} A & b \end{bmatrix}'w_k$ can be trivially asynchronously parallelized on a distributed memory platform using a Fork-join model, when
\begin{enumerate}
\item the rows of $\begin{bmatrix} A & b \end{bmatrix}$ are distributed across the different storages, and
\item the \texttt{RPMStrategy()} generates $\lbrace w_k : k+1 \in \mathbb{N} \rbrace$ such that $w_k$ have independent groups of components (e.g., the Gaussian Strategy and the Count-Sketch strategy).
\end{enumerate}
When these two requirements are met, each processor can generate its own $\lbrace w_k : k+1 \in \mathbb{N} \rbrace$ and operate on the local rows of $\begin{bmatrix} A & b \end{bmatrix}$. It can then simply pass the resulting row to the master processor which performs the iterate and $S_k$ matrix updates. For each iteration, a scattering and gathering of the data is performed but no other data exchange is required.

\Cref{table:full-memory-parallel-comparison} summarizes the time and total computational costs of computing $x_k$ and $S_k$ from $x_0$ and $S_0$ in the following context: (1) the sequential platform refers to the case where there is a single processor with a sufficiently large memory to store the system, and perform the necessary operations in \cref{alg: rank-one RPM full mem};
(2) the shared memory platform assumes that there are $p+1$ processors that share a sufficiently large memory. One of the processors is dedicated to performing the iterate and matrix updates, while the remaining $p$ processors compute $\begin{bmatrix} A & b \end{bmatrix}'w_k$;
(3) the distributed memory architecture assumes that there are $p+1$ processors each with a sufficient memory capacity. The rows of $\begin{bmatrix} A & b \end{bmatrix}$ are split evenly or nearly evenly amongst $p$ of the processors, and each process only manipulates its local information about $A$ and $b$. Finally, master processor is dedicated to performing the iterate and matrix updates.

\begin{table}[!htb]
\centering
\renewcommand{\arraystretch}{1.3}
\begin{tabular}{@{}lrcrcrcrcrc@{}} \toprule
\multicolumn{11}{c}{\textbf{Total Time and Effort Costs to Iteration $k$}} \\ \midrule 
\multicolumn{1}{c}{Platform} & & \multicolumn{3}{c}{Computing $\begin{bmatrix} A & b \end{bmatrix}'w$} & & \multicolumn{3}{c}{Update Costs} & & Network \\ 
\cmidrule{3-5}
\cmidrule{7-9}
  & & Time & & Total Effort & & Iterate & & Matrix & & \\ \midrule 
Sequential & & $\bigO{knd}$ & & $\bigO{knd}$ & & $\bigO{kd^2}$ & & $\bigO{kd^3}$ & & No \\
Shared Memory & & $\bigO{knd/p}$ & & $\bigO{knd}$ & & $\bigO{kd^2}$ & & $\bigO{kd^3}$  & & No \\
Distributed Memory & & $\bigO{knd/p^2}$ & & $\bigO{knd/p}$ & & $\bigO{kd^2}$ & & $\bigO{kd^3}$ & & Yes \\ \bottomrule
\end{tabular}
\caption{A summary of the time and total computational cost (effort) incurred by \cref{alg: rank-one RPM full mem} and its parallelized variants. We do not report any advantages that should be exploited when $A$ or $w$ are sparse. In the shared and distributed memory platforms, we assume that there are $p$ processors dedicated to computing $A'w$ and $b'w$, and one processor dedicated to computing the updates. The ``Network'' column refers to whether communication costs over a network are incurred.}
\label{table:full-memory-parallel-comparison}
\end{table}

\subsubsection{Memory-Reduced Procedure} \label{subsection:low-memory}

\begin{algorithm}[!hbt]
\KwData{Initialization $x_0$,  \texttt{RPMStrategy()} for $w_0,w_1,\ldots$, \texttt{TerminationCriteria()},\\ Memory Storage Parameter $m$}
\KwResult{Estimate $\hat{x}$}
\BlankLine
$k,j \leftarrow 0, 0$ \\
$\mathcal{S} \leftarrow \emptyset$ \medskip

\While{TerminationCriteria() == false}{
\tcp{Compute search direction}
Generate $w_k$ \\
$q_k \leftarrow A'w_k$ \\
Compute $u_k$ using \cref{alg:modified-GS} on $q_k$ with vectors in $\mathcal{S}$ \\
\If{$u_k == 0$}{ $k \leftarrow k+1$ \\ continue to next iteration } \medskip

\tcp{Update iterate}
$r_k \leftarrow b'w_k - q_k'x_k$ \\
$x_{k+1} = x_k + u_kr_k/(u_k'q_k)$ \medskip

\tcp{Update Memory Storage}
$z_{k+1} \leftarrow u_k/\norm{u_k}$ \\  
\uIf{$j == m$}{
Remove $z_{k+1-m}$ from $\mathcal{S}$ and append $z_{k+1}$ to $\mathcal{S}$
}\Else{
Append $z_{k+1}$ to $\mathcal{S}$ \\
$j \leftarrow j+1$
} \medskip

\tcp{Update Iteration Counter}
$ k \leftarrow k+1$

}
\Return{$x_{k+1}$}
\caption{Rank-One RPM Method with Partial Orthogonalization \label{alg: rank-one RPM low mem}}
\end{algorithm}

Another notable aspect of \cref{alg: rank-one RPM full mem} (and its aforementioned parallel variants described above) is that it must store and manipulate the matrix $S_k$ at each iteration, which is clearly expensive when $d$ is large or is excessive when $d^3$ is comparable to $n$ or greater than $n$. This difficulty motivates a partial orthogonalization approach, as described in \cref{alg: rank-one RPM low mem}. In this approach, a user-defined parameter $m < d$ specifies the number of $d$-dimensional vectors needed to implicitly store an approximate representation of $S_k$ (based on \cref{theorem: S are orthogonal projections}). With this implicit representation, the cost of computing $u_k$ reduces to $\bigO{md}$,\footnote{If $q_k$ replace $u_k$ in the calculation of $z_k$, then the cost of computing $u_k$ is $\bigO{dm^2}$ \citep[see][Ch. 5.2]{golub2012}.} which, consequently, reduces the overall cost of updating $x_{k}$ to $x_{k+1}$ to $\bigO{md}$. Moreover, because $S_k$ is implicitly represented by a $m$ $d-$dimesnional vectors in $\mathcal{S}$, there is no notable additional computational cost incurred for updating $S_{k}$ to $S_{k+1}$. Thus, an entire iteration incurs a computational cost $\bigO{md}$ plus the cost of computing $\begin{bmatrix} A & b \end{bmatrix}'w_k$, which can be mollified under the strategies above in shared memory or distributed memory platforms.

\begin{remark} 
\cref{alg: rank-one RPM low mem} is an efficient implementation of the partial orthogonalization procedure and, as a result, at $m=0$, seems to only recover row-action base randomized iterative methods as specified by \cref{eqn:no-memory-iteration}. A less efficient algorithm based on directly applying \cref{eqn: gain rank-one update,eqn: rank-one update-matrix} with the appropriate low memory modification would recover all rank-one base randomized iterative methods when $m=0$.
\end{remark}

\begin{algorithm}[hbt]
\KwData{Vector $q_k$,  Orthonormal Set $\lbrace z_1,\ldots,z_{k-1} \rbrace$}
\KwResult{Projection $q_k$ onto subspace orthogonal to $\lbrace z_1,\ldots,z_{k-1} \rbrace$}
\BlankLine
$j \leftarrow 0$ \\
$t_0 \leftarrow q_k$ \\

\While{j $\leq k-1$}{
$j \leftarrow j+1$\\
$t_{j} \leftarrow t_{j-1} - (z_j't_{j-1}) z_j$
}
\Return{$t_{k}$}
\caption{Modified Gram-Schmidt \label{alg:modified-GS}}
\end{algorithm}

\subsubsection{Optimizing Communication Overhead. Structured Systems} \label{subsubsection:structured}

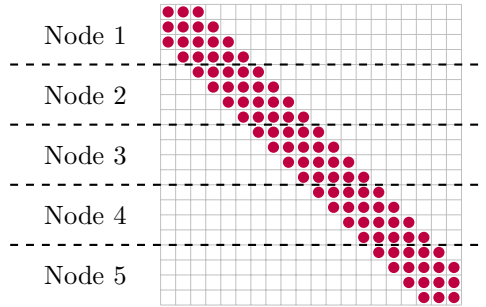
\begin{figure}[hbt]
\centering
\begin{tikzpicture}
\def\q{2};
\def\dim{4};

\draw[step=0.2cm,gray,very thin, opacity=0.5] (0,0) grid (\dim,\dim);

\pgfmathparse{\dim/0.2-1};
\let\size\pgfmathresult;

\foreach \i in {0,...,\size}{
\foreach \j in {0,...,\size}{
	\pgfmathparse{abs(\i - \j)};
	\ifdim \pgfmathresult cm > \q cm \relax 
	\else {
		\pgfmathparse{\i*0.2+0.1};
		\let\x\pgfmathresult;
		
		\pgfmathparse{\dim-\j*0.2-0.1};
		\let\y\pgfmathresult;
		\node at (\x,\y) [circle, fill=purple, inner sep=1.5pt] {};
	}\fi;

};
};

\def\cpu{5};

\foreach \p in {2,...,\cpu}{
	\pgfmathparse{\dim/(\cpu)*(\p-1)};
	\let\y\pgfmathresult;
	\draw[dashed,thick] (-2,\y) -- (4.5,\y);
};


\foreach \p in {1,...,\cpu}{
	\pgfmathparse{\dim - \dim/(\cpu)*(\p-0.5)};
	\let\y\pgfmathresult;
	\node at (-1,\y) {Node \p};
};
\end{tikzpicture}
\caption{A representation of a $20 \times 20$ banded matrix with bandwidth $\tilde{Q} +1 = 5$, whose rows are split across five compute nodes (represented by the dashed line). Note, the empty grid points represent zeros, while the filled grid points represent nonzero values.}
\label{figure:tikz-banded-matrix}
\end{figure}

In the above approaches, we take for granted that $d$ is not so large such that communicating $\bigO{d}$ vectors is acceptable during the procedure. However, for many problems coming from the solution of differential equations \citep[e.g., see][]{dongarra1986}, $d$ and $n$ are of the same order and are so large that communicating $\bigO{d}$ vectors at arbitrary points during the procedure is impossible. Fortunately, linear system problems in this class are highly sparse and structured \citep[][Ch. 2]{saad2003}. A simple example is the case where $A$ is a square, banded system with nonzero bandwidth $\tilde{Q}+1$ for some $\tilde{Q} \ll n = d$; that is, $A_{ij} = 0$ if $|i-j| > \tilde{Q}$ and the remaining $A_{ij}$ can take arbitrary values. 

For such sparse and structured problems, our methodology can be efficiently implemented across a distributed memory platform with $p$ processors under some additional qualifications. However, to understand these qualifications, let us first introduce some notation and concepts that define the communication pattern across the $p$ nodes. 

Suppose somehow that we distribute the equations of our linear system of interest across $p$ nodes. \cref{figure:tikz-banded-matrix} shows how the coefficient matrix of a $20 \times 20$ banded system with bandwidth $5$ can be distributed across five nodes. Note, in this example, the entries of the constant vector would be stored on the same processor as the corresponding rows of the coefficient matrix. 
Moreover, we need a way of tracking which components of $x$ are manipulated by each node: 
let $\mathcal{X}_i$ be the set of indices of the components of $x$ with nonzero coefficients at node $i$ in the distributed system for $i=1,\ldots,p$.
In our example, $\mathcal{X}_1 = \lbrace 1,\ldots,6 \rbrace$, $\mathcal{X}_2 = \lbrace 3,\ldots,10 \rbrace$, $\mathcal{X}_3 = \lbrace 7,\ldots,14 \rbrace$, $\mathcal{X}_4 = \lbrace 11,\ldots,18 \rbrace$, and $\mathcal{X}_5 = \lbrace 15,\ldots,20 \rbrace$. 
Finally, for any vector $z$ and any set $\mathcal{X}$ over the indices of $z$, let $z[\mathcal{X}]$ be the vector whose elements are the elements of $z$ indexed by $\mathcal{X}$.

From this example and from our discussion in \cref{subsubsection:parallel} of distributing the \texttt{RPMStrategy()}, we can use the local rows of $A$ at Node 1 and a Gaussian sketch to generate a $q_1 \in \mathbb{R}^d$ such that $q_1[\lbrace 1,\ldots, 6 \rbrace]$ are arbitrarily valued and $q_1[\lbrace 7,\ldots,20 \rbrace] = 0$. Thus, our vector $q_k$ is highly sparse and can be generated locally on the node. However, following \cref{alg: rank-one RPM full mem}, the next step of computing $u_k$ requires computing the product between $S_k$ and $q_k$, which, in a naive implementation, would require storing a dense $d \times d$ matrix $S_k$ and computing a global matrix-vector product. Such a required computation raises several concerns, which we detail and address in the following enumeration.
\begin{enumerate}
\item Given that $d$ is relatively large to the computing environment, is storing a $d \times d$ matrix even feasible? Generally, the answer will be that storing such a matrix is infeasible. However, by exploiting the properties of $S_k$ (see \cref{theorem: S are orthogonal projections}), we will approximately and implicitly store $S_k$ as $\mathcal{S}$, which is a collection of orthonormal vectors. 
\item Even if we use $\mathcal{S}$ in place of $S_k$, will the resulting implicit matrix-vector product and update of $\mathcal{S}$ incur prohibitive communication costs? To answer these questions completely, we will need to specify how the implicit matrix-vector product will be computed and how $\mathcal{S}$ will be stored. Here, we will compute the implicit matrix-vector product by using twice-iterated classical Gram-Schmidt (\cref{alg:iterated-GS}), which was shown to be numerically stable in the seminal work of \cite{giraud2005}. Owing to this calculation pattern, we can store $\mathcal{S}$ in a distributed fashion across the $p$ processors, which we detail below along with the communication cost of the synchronization of $\mathcal{S}$. 
\end{enumerate}

\begin{algorithm}[hbt]
\KwData{Vector $q_k$,  Orthonormal Set $\lbrace z_1,\ldots,z_{k-1} \rbrace$}
\KwResult{Projection $q_k$ onto subspace orthogonal to $\lbrace z_1,\ldots,z_{k-1} \rbrace$}
\BlankLine

$t_0 \leftarrow q_k$ \\

\For{$j=0:1$}{
\tcp{Compute projection onto Orthogonal Set}
\tcp{$\mathcal{P}$ is a set of vectors}
$\mathcal{P} \leftarrow \mathrm{map}(l \mapsto (t_j'z_l) z_l, l=1:k-1)$\\ \medskip

\tcp{Compute orthogonal component}
\tcp{$\mathrm{sum}$ sums over the set $\mathcal{P}$}
$t_{j+1} \leftarrow t_j  - \mathrm{sum}(\mathcal{P})$ \\

}

\Return{$t_{2}$}
\caption{Twice Iterated Gram-Schmidt \label{alg:iterated-GS}}
\end{algorithm}

To understand the costs associated with computing $u$ from the orthonormal vectors in $\mathcal{S}$ and the vector $q$, we will characterize  the support of $u$ (i.e., index set of its nonzero entries). 

\begin{lemma} \label{lemma:support-GS}
Let $q \in \mathbb{R}^d$ and let $\mathcal{Q} = \lbrace i : q[i] \neq 0 \rbrace \subset \lbrace 1,\ldots, d \rbrace$. Let $\lbrace z_1,\ldots,z_m \rbrace \subset \mathbb{R}^d$ be a set of orthonormal vectors (hence, $m \leq d$), and let $\mathcal{Z}_j = \lbrace i : z_j[i] \neq 0 \rbrace \subset \lbrace 1,\ldots, d \rbrace$ for $j=1,\ldots,m$. If $u$ denotes the result of \cref{alg:iterated-GS} applied to $q$ over the set $\lbrace z_1,\ldots,z_m \rbrace$ then 
\begin{equation}
\mathcal{U} := \lbrace i : u[i] \neq 0 \rbrace \subset \left(\bigcup_{j \in \mathfrak{Q}} \mathcal{Z}_j \right) \cup \mathcal{Q},
\end{equation}
where $\mathfrak{Q} = \lbrace j : \mathcal{Q} \cap \mathcal{Z}_j \neq \emptyset \rbrace \subset \lbrace 1,\ldots,m \rbrace$.
\end{lemma}
\begin{proof}
Letting $Z$ denote the matrix whose columns are elements of the orthonormal set, we recall that classical Gram-Schmidt generates $u = (I_d - ZZ')q$. Thus, twice iterated Gram-Schmidt can be written as
\begin{equation}
(I_d - ZZ')(I_d - ZZ')q = (I_d - 2ZZ' + ZZ') q = (I_d - ZZ')q = u,
\end{equation}
which is expected in exact arithmetic. Thus, we can consider classical Gram-Schmidt and ignore the iteration to compute the support of $u$. For any $l = 1,\ldots,d$,
\begin{equation} \label{eqn:GS-component}
u[l] = q[l] - \sum_{j=1}^k (q' z_j) z_j[l] = q[l] - \sum_{j \in \mathfrak{Q}} (q'z_j) z_j[l],
\end{equation}
where we use the fact that if $j \not\in \mathfrak{Q}$ then $q' z_j = \sum_{l \in \mathcal{Q} \cap \mathcal{Z}_j } q[l] z_j[l] = \sum_{l \in \emptyset} q[l] z_j[l] = 0$. 
For a contradiction, suppose $l \in \mathcal{U}$ such that 
\begin{equation}
l \not\in \left(\bigcup_{j \in \mathfrak{Q}} \mathcal{Z}_j \right) \cup \mathcal{Q}.
\end{equation}
Then, $q[l] = 0$ and $z_j[l] = 0$ for $j \in \mathfrak{Q}$. Using the above formula for $u[l]$, $u[l] = 0 - \sum_{j \in \mathfrak{Q}} (q'z_j) 0 = 0$, 
which is a contradiction. 
\end{proof}

At iteration $k$, \cref{lemma:support-GS} states that the support of $u_k$ will depend on the support of $\lbrace z_m,\ldots,z_1 \rbrace$, which, in turn, has elements whose support depend on (a subset of) $\lbrace u_{k-1},\ldots,u_0 \rbrace$. Moreover, if $\lbrace u_{k-1},\ldots,u_0 \rbrace$ has elements whose combined support cover $\lbrace 1,\ldots,d \rbrace$, which will be necessary to solve the system,\footnote{Note, if the combined supports of the elements of $\lbrace u_{k-1},\ldots,u_0 \rbrace$ do not cover all of $\lbrace 1,\ldots,d\rbrace$, then some components of our iterates, $\lbrace x_k \rbrace$ will not be updated.} it is possible that the support of $u_k$ will be all of $\lbrace 1,\ldots,d \rbrace$ (ignoring any trivial independence in the system). Thus, it appears that we will eventually have to store vectors in $\mathcal{S}$ whose support is all of $\lbrace 1,\ldots, d \rbrace$. Naively, we may think that we need a faithful copy of $\mathcal{S}$ at each node in the system, which incurs prohibitive communication costs as the support of $u_k$ tends to $\lbrace 1,\ldots, d \rbrace$. While this is true, a careful inspection of Gram-Schmidt and the nonzero patterns of $q_k$ suggest a less naive approach, which we now detail.

We begin by supposing that on a processor $i \in \lbrace 1,\ldots, p \rbrace$, only $z_j[\mathcal{X}_{i}]$ are stored on the node for every $j=1,\ldots,k$. Immediately, we have eliminated the need for synchronizing all of $\mathcal{S}$ on each processor. Instead, we need only to synchronize those components of $z_j$ in $\mathcal{X}_i \cap \mathcal{X}_j$ for all $i \neq j$. Thus, we have that our synchronization costs will depend on the maximum overlap, $Q$, between two processors, which, formally, is 
\begin{equation}
Q = \max_{i \neq j} | \mathcal{X}_i \cap \mathcal{X}_j|.
\end{equation}

Now, we can understand the precise nature of this synchronization by inspecting \cref{alg:iterated-GS}. If for some $j = 1,\ldots,p$, $q_k[\mathcal{X}_j^c] = 0$, then 
\begin{equation} \label{eqn:IGS-loop-1}
t_1[l] = \begin{cases}
- \sum_{t=1}^{m} \left( \sum_{r \in \mathcal{X}_j} q_k[r] z_t[r] \right) z_t[l] & \forall l \in \mathcal{X}_j^c \\
q_k[l] - \sum_{t=1}^{m} \left( \sum_{r \in \mathcal{X}_j} q_k[r] z_t[r] \right) z_t[l] & \forall l \in \mathcal{X}_j 
\end{cases}
\end{equation}
From \cref{eqn:IGS-loop-1}, we see that we must communicate the values of $q[l]$ to all nodes $i \in \lbrace 1,\ldots,p \rbrace \setminus \lbrace j \rbrace$ such that $\mathcal{X}_j \cap \mathcal{X}_i \neq \emptyset$, and we must communicate the $m$ inner products to all $p-1$ nodes. The resulting number of floating point values that must be communicated (counting each replicate to a node individually) during the first iteration of \cref{alg:iterated-GS} is
\begin{equation} 
\sum_{i \in \mathfrak{Q}_j\setminus\lbrace j \rbrace} |\mathcal{X}_j \cap \mathcal{X}_i| + m(p-1),
\end{equation}
where $\mathfrak{Q}_j = \lbrace i : \mathcal{X}_i \cap \mathcal{X}_j \neq \emptyset \rbrace$ for $j=1,\ldots,p$ (see the notation in \cref{lemma:support-GS}). For the second iteration of \cref{alg:iterated-GS}, we must broadcast $m$ inner products that are partially computed (using some ordering that respects the non-associative property of floating point complexity) on each node to the remaining $p-1$ nodes. Thus, the number of floating point values that must be communicated (counting each replicate to a processor individually) to ensure synchronization is
\begin{equation} \label{eqn:communication-complexity}
\sum_{i \in \mathfrak{Q}_j\setminus\lbrace j \rbrace} |\mathcal{X}_j \cap \mathcal{X}_i| + m(p-1) + mp(p-1),
\end{equation}
which we can bound by
\begin{equation}
Q (F-1) + m(p^2-1), \text{ where } F = \max_{j} |\mathfrak{Q}_j|.
\end{equation}
Noting that $Q$ represents the maximum shared indices between two nodes and that $F$ represents the maximum number of nodes that overlap, the first term in the bound can be controlled by the ordering choice of the differential equations that generate the system, but a discussion of this topic is beyond the scope of this work. \cref{alg: rank-one RPM low com} summarizes a simple version of the procedure described here. We can also modify this algorithm to the low memory context of \cref{alg: rank-one RPM full mem} by limiting the number of vectors that can be stored in $\mathcal{S}$.

\begin{algorithm}[!hbt]
\KwData{Initialization $x_0$, Distributed \texttt{RPMStrategy()} for $\lbrace w_k \rbrace$, Covering $\lbrace j_k \rbrace \subset \lbrace 1,\ldots,p\rbrace$, \texttt{TerminationCriteria()}, memory storage parameter $m$}
\KwResult{Estimate $\hat{x}$}
\BlankLine
$k,j \leftarrow 0, 0$ \\
$\mathcal{S} \leftarrow \emptyset$ \medskip

\While{TerminationCriteria() == false}{

\For{Node $j_k$}{
\tcp{Compute search direction on node $j_k$}
Generate $w_k$ from distributed \texttt{RPMStrategy()} \\
$q_k \leftarrow A'w_k$ \\
$r_k \leftarrow b'w_k - q_k[\mathcal{X}_{j_k}]'x_k[\mathcal{X}_{j_k}]$ \\ \medskip 

\tcp{First Gram-Schmidt Iteration, note $q_k[\mathcal{X}_{j_k}^c] = 0$}
$\mathcal{I} \leftarrow  \mathrm{map}(z \to q_k[\mathcal{X}_{j_k}]'z[\mathcal{X}_{j_k}], z \in \mathcal{S})$ \\
Communicate inner products in $\mathcal{I}$ to remaining $p-1$ nodes. \\
Communicate $q_k[\mathcal{X}_{j_k} \cap \mathcal{X}_i]$ for $i \neq j_k$. \\
\If{$u_k == 0$}{ $k \leftarrow k+1$ \\ continue to next iteration }
} \medskip

\For{Each Node $j$}{
Compute $t_1[\mathcal{X}_j]$ locally from \cref{eqn:IGS-loop-1} \\ \medskip 

\tcp{Second Gram-Schmidt Iteration} 
$\mathcal{I}_j \leftarrow \mathrm{map}(z \to t_1[\mathcal{X}_j]'z[\mathcal{X}_{j}], z \in \mathcal{S})$ \\
Synchronize inner products in $\mathcal{I}_j$ to remaining $p-1$ nodes for $j=1,\ldots,p$. \\
$u_k[\mathcal{X}_j] \leftarrow t_1[\mathcal{X}_j] - \sum_{z \in \mathcal{S}} \left( \sum_{l=1}^p  t_1[\mathcal{X}_l]'z[\mathcal{X}_l] \right) z[\mathcal{X}_j]$ \\ \medskip 

\tcp{Synchronize and Update $\mathcal{S}$}
Synchronize $\norm{ u_k}_2$ from local computation \\
Locally store $z_{k+1}[\mathcal{X}_j] = u_k[\mathcal{X}_j]/\norm{u_k}_2$ in $\mathcal{S}$. 
} \medskip 

\For{Node $j_k$}{
\tcp{Compute Step Size}
Synchronize $\alpha_k \leftarrow r_k/(u_k[\mathcal{X}_{j_k}]' q_k[\mathcal{X}_{j_k}])$
} \medskip 

\For{Each Node $j$}{
\tcp{Update Iterate}
$x_{k+1}[\mathcal{X}_j] \leftarrow x_k[\mathcal{X}_j] + \alpha_k u_k[\mathcal{X}_j]$
}

\tcp{Update Iteration Counter}
$ k \leftarrow k+1$
}
\Return{$x_{k+1}$}
\caption{Rank-One RPM Method for Limited Communication \label{alg: rank-one RPM low com}}
\end{algorithm}

\section{Convergence Theory for Orthogonalization} \label{section:full-memory}
Here, we prove that the complete orthogonalization approach (i.e., \cref{alg: rank-one RPM full mem}) converges to the solution under a variety of sampling RPM strategies.
In \cref{subsection:core-results-full-memory}, we establish a collection of core results that are useful in characterizing the behavior of our procedure. 
A key feature of these core results is that they will rely on a stopping time $T$, which will depend on the random variables $\lbrace w_k \rbrace$.
Therefore, in \cref{subsection:sampling-times}, we characterize $T$ under common probabilistic relationships between the elements of $\lbrace w_k \rbrace$.
All statements hold with probability one unless stated otherwise.

\subsection{Core Results} \label{subsection:core-results-full-memory}

We establish two key results. First, we establish that our procedure is an orthogonalization procedure: that is, the matrices $\lbrace S_k \rbrace$ project the current search direction onto a subspace that is orthogonal to previous search directions. Second, we characterize the limit point of our iterates, $\lbrace x_k \rbrace$, in terms of a true solution of the linear system and the subspace generated by the rank-one RPMs, $\lbrace V_k \rbrace$.

\begin{theorem} \label{theorem: S are orthogonal projections}
Let $\lbrace w_l : l + 1 \in \mathbb{N} \rbrace \subset \mathbb{R}^n$ be an arbitrary sequence in $\mathbb{R}^n$, and let $\mathcal{R}_0 = \lbrace 0 \rbrace \subset \mathbb{R}^d$ and $\mathcal{R}_{l} = \linspan{A'w_0,\ldots,A'w_{l-1}}$ for $l \in \mathbb{N}$. Now, let $S_0 = I_d$ and $\lbrace S_l : l \in \mathbb{N} \rbrace$ be defined recursively as in \cref{eqn: rank-one update-matrix}. Then, for $l \geq 0$, $S_l$ is an orthogonal projection matrix onto $\mathcal{R}_l^\perp$. 
\end{theorem}
\begin{proof}
We will prove the result by induction. For the base case, $l=0$, $S_0 = I_d$. It follows that $S_0$ is an orthogonal projection onto $\mathcal{R}_0^\perp = \mathbb{R}^d$ since $S_0^2 = I_d^2 = I_d = S_0$ and $\range{I_d} = \mathbb{R}^d$.
Now suppose that the result holds for $l > 0$. If $S_l A'w_l = 0$ then there is nothing to show. Therefore, for the remainder of this proof, suppose $S_l A' w_l \neq 0$. 

First, we show that $S_{l+1}$ is a projection matrix by verifying that $S_{l+1}^2 = S_{l+1}$ by direct calculation. Making use of the recursive definition of $S_{l+1}$ and the induction hypothesis that $S_l^2 = S_l$, 
\begin{equation}
\begin{aligned}
S_{l+1}^2 &= \left( S_l - \frac{S_l A' w_l w_l' A S_l}{w_l'A S_l A' w_l}\right) \left( S_l - \frac{S_l A' w_l w_l'A S_l}{w_l'A S_l A' w_l}\right) \\
		  &= \left( S_l - \frac{S_l A' w_l w_l' A S_l}{w_l'A S_l A' w_l}\right)\left( I_d - \frac{A'w_lw_l'A S_l}{w_l'A S_l A' w_l} \right) \\
		  &= S_l - 2\frac{S_l A' w_l w_l' A S_l}{w_l'A S_l A' w_l} + \frac{S_l A' w_l w_l' A S_l}{w_l'A S_l A' w_l} = S_{l+1}.
\end{aligned}
\end{equation}

Second, we use the fact that a projection is orthogonal if and only if it is self-adjoint to show that $S_{l+1}$ is an orthogonal projection. By induction, because $S_l$ is an orthogonal projection, $S_l' = S_l$, and so
\begin{equation}
S_{l+1}' = S_l' - \frac{S_l A' w_l w_l' A S_l}{w_l'A S_l A' w_l} = S_{l+1}.
\end{equation}

Finally, let $v$ be in the range of $S_{l+1}$ and we can decompose $v$ into the components $u$ and $y$ such that $v = u+y$, $0 = u'y$ and $y \in \mathcal{R}_{l+1}$. We will show that $y = 0$, which characterizes the range of $S_{l+1}$ as being all vectors orthogonal to $\mathcal{R}_{l+1}$. To show this note that because $S_{l+1}$ is a projection matrix, we have that 
\begin{equation} \label{eqn-proof-S-ortho:1}
u + y = v = S_{l+1} v = S_{l+1}u + S_{l+1} y.
\end{equation}
By construction $\mathcal{R}_l \subset \mathcal{R}_{l+1}$ and so $u \in \mathcal{R}_l^{\perp}$. Using the induction hypothesis, we then have that $S_l u = u$. Moreover, because $u \in \mathcal{R}_{l+1}^\perp$ by construction, $u'A'w_l = 0$. Then, using the recursive definition of $S_{l+1}$, we have that
\begin{equation}
S_{l+1} u = S_l u - \frac{S_l A' w_l w_l' A S_lu}{w_l'A S_l A'w_l} = u - \frac{S_l A' w_l w_l' A u}{w_l'A S_l A'w_l} = u.
\end{equation}
Therefore, $u = S_{l+1}u$ and, by \cref{eqn-proof-S-ortho:1}, $y = S_{l+1} y$. We now decompose $y$ into $y_1$ and $y_2$ where $y_1 \in \mathcal{R}_l$ and $y_2 \in \mathcal{R}_l^\perp \cap \mathcal{R}_{l+1}$. By the induction hypothesis, $\mathcal{R}_l^\perp \cap \mathcal{R}_{l+1} = \linspan{ S_{l} A' w_l }$. Therefore, $S_l y = y_2$ and $\exists \alpha \in \mathbb{R}$ such that $y_2 = \alpha S_l A' w_l$. Finally, using the recursive formulation of $S_{l+1}$ and $S_l y = y = \alpha S_l A' w_l$, 
\begin{equation}
y = S_{l+1} y = S_l y - \frac{S_l A' w_l w_l' A S_l y}{w_l' A S_l A' w_l} = \alpha S_l A' w_l  - \alpha S_l A' w_l = 0.
\end{equation}
Thus, we have shown that the range of $S_{l+1}$ is orthogonal to $\mathcal{R}_{l+1}$. 
\end{proof}

From \cref{theorem: S are orthogonal projections}, we see that our procedure is an orthogonalization procedure just like quasi-Newton methods \citep[][Ch. 8]{nocedal2006} and conjugated direction methods \citep{hestenes2012}. As a consequence, we have the following common and insightful characterization of the iterates of such an orthogonalization procedure.

\begin{corollary} \label{corollary: orthogonalization-iterate-subspace}
In addition to the setting of \cref{theorem: S are orthogonal projections}, let $x_0 \in \mathbb{R}^d$ be arbitrary and let $\lbrace x_{l} : l \in \mathbb{N} \rbrace$ be defined according to \cref{eqn: rank-one update-param}. For any $l \geq 0$, $x_{l+1} \in \linspan{x_0,A'w_0,\ldots,A'w_l}$.
\end{corollary}
\begin{proof}
We again proceed by induction. Because $S_0 = I_d$, the case of $x_1$ follows by recursion formula, \cref{eqn: rank-one update-param}. Now suppose that the result holds up to some $l > 0$. Note, by the recursion formula
\begin{equation}
x_{l+1} = x_l + \gamma S_l A' w_l, \quad \text{where} \quad \gamma = \begin{cases} \frac{w_l'(b - Ax_l)}{w_l' A S_l A'w_l} & S_l A' w_l \neq 0 \\
0 & \text{otherwise.} 
\end{cases}
\end{equation}
Therefore, $x_{l+1} \in \linspan{ x_l, S_l A' w_l}$. Now, using the induction hypothesis, 
\begin{equation}
\linspan{ x_l, S_l A' w_l } \subset \linspan{x_0, A'w_0,\ldots,A'w_{l-1}, S_l A' w_l }.
\end{equation}
Second, when $S_l A' w_l = 0$, then $A'w_l \in \mathcal{R}_{l}$. Consequently,
\begin{equation}
x_{l+1} \in \linspan{x_0, A'w_0,\ldots,A'w_{l-1}} = \linspan{x_0,A'w_0,\ldots,A'w_l}.
\end{equation}
Now suppose $S_l A' w_l \neq 0$. By \cref{theorem: S are orthogonal projections}, $S_l$ is an orthogonal projection onto $\linspan{A'w_0,A'w_1,\ldots,A'w_{l-1}}^{\perp}$. 
Hence,
$x_{l+1} \in \linspan{ x_l, S_l A' w_l}$, which is contained in $ \linspan{x_0, A'w_0,\ldots,A'w_l}.$
\end{proof}

\Cref{corollary: orthogonalization-iterate-subspace} demonstrates that, as is common with orthogonalization procedures, the iterates are in a subspace generated by the initial iterate and the search directions $\lbrace A'w_0,\ldots,A'w_l \rbrace$. For deterministic procedures, such a characterization is usually sufficient and the next step would be to demonstrate that the iterates are the closest points to the true solutions within the given subspace. However, for a procedure in which the subspace is randomly generated, there is substantially more nuance. In order to be conscientious of space, we will not go through the litany of issues, but rather skip to the appropriate definitions and characterizations.

First, we begin by defining the maximal possible subspace that can be generated by a random quantity $A'w$. Let $w \in \mathbb{R}^n$ be a random variable defined on a space $\Omega$, and let
\begin{equation} \label{eqn: row span}
\mathcal{N}(w) = \linspan{ z \in \mathbb{R}^d: \Prb{z'A'w = 0} = 1} \text{ and } \mathcal{R}(w) = \mathcal{N}(w)^\perp.
\end{equation}
Moreover, we define the subspace $\mathcal{V}(w)$ such that $\mathcal{V}(w) \perp \mathcal{R}(w)$ and $\mathcal{V}(w) + \mathcal{R}(w) = \mathrm{row}(A)$ (hence, $\mathcal{V}(w) \oplus \mathcal{R}(w) = \mathrm{row}(A)$). Correspondingly, let $P_W$ denote the orthogonal projection matrix onto a subspace $W \subset \mathbb{R}^d$. The following result characterizes $\mathcal{R}(w)$.
\begin{lemma} \label{lemma: characterize R_w}
For $\mathcal{R}(w)$ as defined in \cref{eqn: row span}, $\mathcal{R}(w)$ is the smallest subspace of $\mathbb{R}^d$ such that $\Prb{ A'w \in \mathcal{R}(w) } = 1$. 
\end{lemma}
\begin{proof}
First, we verify that $\Prb{A'w \in \mathcal{R}(w)} = 1$. Suppose that $\Prb{A'w \in \mathcal{R}(w)} < 1$. Then, 
\begin{equation}
\Prb{\exists z \perp \mathcal{R}(w): z'A'w \neq 0} > 0.
\end{equation}
However, we know that for any $z$ such that $z \perp \mathcal{R}(w)$, $z \in \mathcal{N}(w)$ and $z'A'w = 0$ with probability one, which is a contradiction. Hence, $\Prb{A'w \in \mathcal{R}(w)} = 1$. 

Now suppose there is a proper subspace of $\mathcal{R}(w)$, $U$, such that $\Prb{ A'w \in U} = 1$. Let $U^{\perp \mathcal{R}(w)}$ denote the subspace orthogonal to $U$ relative to $\mathcal{R}(w)$. Then, $\Prb{z' A'w = 0 } = 1$ for any $z \in U^{\perp \mathcal{R}(w)}$, which implies that $U^{\perp \mathcal{R}(w)} \subset \mathcal{N}(w)$. However, since $U^{\perp \mathcal{R}(w)} \subset \mathcal{R}(w) \perp \mathcal{N}(w)$, $U^{\perp \mathcal{R}(w)} = \lbrace 0 \rbrace$. Thus, $\mathcal{R}(w)$ is the smallest subspace such that $\Prb{A'w \in \mathcal{R}(w)} = 1$.
\end{proof}

Second, we must define when the maximal possible subspace of $A'w$ can be achieved by a sequence of random variables $\lbrace A'w_0,\ldots,A'w_l \rbrace$, which may or may not be related to $A'w$. Note, by not requiring a relationship between $\lbrace A'w_0,\ldots,A'w_l \rbrace$ and $A'w$ our next result is particularly general and applies to a variety of situations, from the case in which $\lbrace w_l \rbrace$ are independent copies of $w$ to the case where $\lbrace w_l \rbrace$ have complex dependencies. Now, let $\lbrace w_l : l +1 \in \mathbb{N} \rbrace \subset \mathbb{R}^n$ be random variables defined on $\Omega$, and let $T$ be a stopping time defined by
\begin{equation} \label{eqn: stopping time}
T = \min \lbrace k \geq 0 : \linspan{A'w_0,\ldots,A'w_k} \supset \mathcal{R}(w) \rbrace.\footnote{Below we will assume that $A'w \in \mathcal{R}(w)$ with probability one. If we relax this, this will change the results in a predictable manner but will require additional notation. To avoid such notation, we will leave this more general case to future work if there is a sampling case that merits it.}
\end{equation}
Using this notation, we have the following fundamental characterization result of the limit points of $\lbrace x_l \rbrace$. 

\begin{theorem} \label{theorem: terminal iteration characterization}
Let $w$ be a random variable, and let $\mathcal{R}(w)$, $\mathcal{N}(w)$ and $\mathcal{V}(w)$ be as defined above (see \cref{eqn: row span}).
Moreover, let $w_0, w_1,\ldots \in \mathbb{R}^n$ be random variables such that $\Prb{A'w_l \in \mathcal{R}(w)}=1$ for all $l+1 \in \mathbb{N}$, and let $T$ be as defined in \cref{eqn: stopping time}. Let $x_0 \in \mathbb{R}^d$ be arbitrary and $S_0 = I_d$, and let $\lbrace x_l : l \in \mathbb{N} \rbrace$ and $\lbrace S_l : l \in \mathbb{N} \rbrace$ be defined as in \cref{eqn: rank-one update-param,eqn: rank-one update-matrix}. On the event $\lbrace T < \infty \rbrace$, 
\begin{enumerate}
\item For any $s \geq T+1$, $S_{T+1} = S_s$ and $x_{T+1} = x_s$.
\item  If $Ax=b$ admits a solution $x^*$ (not necessarily unique), then 
\begin{equation}
x_{T+1} = P_{\mathcal{N}(w)} x_0 + P_{\mathcal{R}(w)} x^*.
\end{equation}
\end{enumerate}
\end{theorem}
\begin{proof}
Recall that $\mathcal{R}_{k+1} = \linspan{A'w_0,\ldots,A'w_{k} }.$ Therefore, by the definition of $T$, $\mathcal{R}_{T+1} = \mathcal{R}(w)$ on the event that $\lbrace T < \infty \rbrace$. Therefore, by \cref{theorem: S are orthogonal projections}, $S_{T+1}$ is an orthogonal projection onto $\mathcal{N}(w)$ and its null space is $\mathcal{R}(w)$. 

We now proceed by induction. Because $\nullsp{S_{T+1}} = \mathcal{R}(w)$ and $A'w_{T+1} \in \mathcal{R}(w)$ with probability one (by hypothesis), $S_{T+1} A' w_{T+1} = 0$. Therefore, by the recursion equations, \cref{eqn: rank-one update-param,eqn: rank-one update-matrix}, $S_{T+2} = S_{T+1}$ and $x_{T+2} = x_{T+1}$. Suppose now that $S_{T+l} = S_{T+1}$ and $x_{T+l} = x_{T+1}$ for $l > 1$. Again, by hypothesis, $A'w_{T+l} \in \mathcal{R}(w) = \nullsp{S_{T+l}}$. Therefore, $S_{T+l} A' w_{T+l} = 0$. By the recursion equations, \cref{eqn: rank-one update-param,eqn: rank-one update-matrix}, $S_{T+l+1} = S_{T+l} = S_{T+1}$ and $x_{T+l+1} = x_{T+l} = x_{T+1}$.

To establish the second part of the result, we must first establish that for any $l \geq 0$,
\begin{equation}
x_{l+1} - x^* = S_{l+1}(x_0 - x^*).
\end{equation}
We will prove this by induction. For $l=0$,
\begin{equation}
\begin{aligned}
x_1 - x^* &= x_0 - x^* + \frac{S_0 A' w_0 w_0'}{w_0'A S_0 A w_0} (A x^* - Ax_0) \\
		  &= \left( I_d - \frac{S_0A' w_0 w_0'A}{w_0' A S_0 A w_0 } \right) (x_0 - x^*),
\end{aligned}
\end{equation}
by the recursion equations, \cref{eqn: rank-one update-param}. Noting that $S_0 = I_d$ and by using \cref{eqn: rank-one update-matrix}, we conclude that $x_1 - x^* = S_1 (x_0 - x^*)$. Now suppose that this relationship holds for some $l > 0$. Again, using \cref{eqn: rank-one update-param},
\begin{equation}
\begin{aligned}
x_{l+1} - x^* &= x_l - x^* + \frac{S_l A' w_l w_l'}{w_l'A S_l A w_l} (A x^* - Ax_l) \\
		  &= \left( I_d - \frac{S_lA' w_l w_l'A}{w_l' A S_l A w_l } \right) (x_l - x^*).
\end{aligned}
\end{equation}
Using the induction hypothesis, $x_l - x^* = S_l(x_0 - x^*)$ and \cref{eqn: rank-one update-matrix},
\begin{equation}
x_{l+1} - x^* = \left(I_d  - \frac{S_lA' w_l w_l'A}{w_l' A S_l A w_l } \right) S_l (x_0 - x^*) = S_{l+1} (x_0 - x^*).
\end{equation}

With this result established and noting that $S_{T+1}$ is a projection onto $\mathcal{N}(w)$ (i.e., $P_{\mathcal{N}(w)} = S_{T+1}$), on the event $\lbrace T < \infty \rbrace$, 
\begin{equation}
\begin{aligned}
x_{T+1} &= x^* + S_{T+1}(x_0 - x^*) \\
	    &= \left(P_{\mathcal{N}(w)} + P_{\mathcal{R}(w)}\right)x^* + P_{\mathcal{N}(w)} x_0 - P_{\mathcal{N}(w)} x^* \\
	    &= P_{\mathcal{R}(w)} x^* + P_{\mathcal{N}(w)} x_0.
\end{aligned}
\end{equation}
\end{proof}

With \cref{theorem: terminal iteration characterization} in hand, the natural subsequent question is when the limit point of the iterates is actually a solution to the original system. This question is addressed in the following corollary.

\begin{corollary} \label{corollary: criteria system solution}
Under the setting of \cref{theorem: terminal iteration characterization}, on the event $\lbrace T < \infty \rbrace$, $Ax_{T+1} = b$ if and only if $P_{\mathcal{V}(w)} x_0 = P_{\mathcal{V}(w)} x^*$. 
\end{corollary}
\begin{proof}
Recall that $\mathrm{row}(A) \perp \nullsp{A}$. Because $\mathcal{R}(w) \subset \mathrm{row}(A)$, $\mathcal{N}(w) = \mathcal{V}(w) + \nullsp{A}$. Moreover, by the definition of $\mathcal{V}(w) \subset \mathrm{row}(A)$, $\mathcal{V}(w) \perp \nullsp{A}$. Therefore, $P_{\mathcal{N}(w)} = P_{\nullsp{A}} + P_{\mathcal{V}(w)}$. Now, using the characterization in \cref{theorem: terminal iteration characterization}, 
\begin{equation}
A x_{T+1} = A P_{\nullsp{A}} x_0 + A P_{\mathcal{V}(w)} x_0 + A P_{\mathcal{R}(w)} x^* = A P_{\mathcal{V}(w)} x_0 + A P_{\mathcal{R}(w)} x^*. 
\end{equation}
Similarly, because $I_d = P_{\nullsp{A}} + P_{\mathcal{V}(w)} + P_{\mathcal{R}(w)}$, 
\begin{equation}
b = Ax^* = A P_{\nullsp{A}} x^* + A P_{\mathcal{V}(w)} x^* + A P_{\mathcal{R}(w)}x^* = A P_{\mathcal{V}(w)} x^* + A P_{\mathcal{R}(w)}x^*.
\end{equation}
Setting these two quantities equal to each other, we conclude that $Ax_{T+1} = b$ if and only if $A P_{\mathcal{V}(w)} x^* = A P_{\mathcal{V}(w)} x_0$. Clearly, if $P_{\mathcal{V}(w)} x_0 = P_{\mathcal{V}(w)} x^*$ then $Ax_{T+1} = b$. So, what we have left to show is that $A P_{\mathcal{V}(w)} x^* = A P_{\mathcal{V}(w)} x_0$ implies
$P_{\mathcal{V}(w)} x_0 = P_{\mathcal{V}(w)} x^*$.

Let $A^+$ denote the Moore-Penrose pseudo-inverse of $A$, and recall that $A^+ A$ is a projection onto $\mathrm{row}(A)$. Moreover, $\mathrm{range}(P_{\mathcal{V}}) \subset \mathrm{row}(A)$. Therefore, since if $Ax_{T+1} =b$ then $AP_{\mathcal{V}(w)} x_0 = AP_{\mathcal{V}(w)} x^*$, if $A x_{T+1} = b$ then  
\begin{equation} \label{proof-eqn:V-in-row-A}
P_{\mathcal{V}(w)} x_0 = (A^+ A) P_{\mathcal{V}(w)} x_0 = A^+ (A P_{\mathcal{V}(w)} x_0 ) = A^+ A P_{\mathcal{V}(w)} x^* = P_{\mathcal{V}(w)} x^*.  
\end{equation}
\end{proof}

\Cref{corollary: criteria system solution} provides criteria on the initial condition and on $\mathcal{V}(w)$ to determine when our procedure will solve the linear system. However, we would rarely have a way of choosing the initial condition apriori such that the requirement of \cref{corollary: criteria system solution} holds. Thus, the alternative is to design $w$ and $\lbrace w_l \rbrace$ so that $\mathcal{V}(w) = \lbrace 0 \rbrace$, which would guarantee that $Ax_{T+1} = b$ on the event $\lbrace T < \infty \rbrace$. It is worth reiterating that we have made very limited assumptions about the relationships between $w$ and $\lbrace w_l \rbrace$ and amongst $\lbrace w_l \rbrace$. This is important because it allows us to apply the preceding results to a variety of common relationship patterns between $w$ and $\lbrace w_l \rbrace$. In the next subsection, we explore some specific relationships and whether these relationships will result in $\mathcal{V}(w) = \lbrace 0 \rbrace$.

\subsection{Common Sampling Patterns} \label{subsection:sampling-times}

\Cref{theorem: terminal iteration characterization} supplies a general result about the behavior of \textit{any} sampling methodology on the solution of the system using \cref{eqn: rank-one update-matrix,eqn: rank-one update-param}, yet it does not suggest a precise sampling methodology. Generally, the sampling methodology choice will depend on both the hardware environment and the nature of the problem. For example, a random permutation sampling methodology will limit the parallelism achievable in \cref{alg: rank-one RPM low com}. On the other hand, a random permutation sampling methodology might be well-advised in a sequential setting where very little known is about the coefficient matrix $A$. Thus, the precise sampling scheme should depend on the hardware environment and should exploit the structure of the problem. 

Despite this, in practice, there are two general sampling schemes that form a basis for more problem and hardware specific sampling schemes: random permutation sampling and independent and identically distributed sampling. The former sampling pattern is exemplified by randomly permuting the equations of the linear system. More concretely, let $e_1,\ldots,e_n \in \mathbb{R}^n$ be the standard basis; let $w$ be a random variable with nonzero probability on each element of the basis; let $\lbrace w_l \rbrace$ be random variables sampled from $\lbrace e_1,\ldots,e_n \rbrace$ without replacement (until the set is exhausted, then we repopulate the set with its original elements and repeat the sampling without replacement). The following statement provides a simple characterization of this sampling scheme.

\begin{lemma} \label{lemma: rand perm sampling}
Let $\lbrace W_1,\ldots,W_N \rbrace \subset \mathbb{R}^n$. Let $w$ be a random variable such that
\begin{equation}
\Prb{ w = W_j} > 0 \quad j = 1,\ldots,N, \quad\text{and}\quad \sum_{j=1}^N \Prb{ w = W_j } = 1.
\end{equation}
Moreover, let $\lbrace w_l : l + 1 \in \mathbb{N} \rbrace$ be random variables sampled from $\lbrace W_1,\ldots,W_N \rbrace$ without replacement (and once the set is exhausted, we repopulate the set with its original elements and repeat sampling without replacement). Then $T \leq N-1$. Moreover, $Ax_{T+1} = b$ for every initialization if $\linspan{A'W_1,\ldots,A'W_N} = \mathrm{row}(A)$, which holds if $\linspan{ W_1,\ldots,W_N}= \mathbb{R}^n$.
\end{lemma}
\begin{proof}
First, note that $\mathcal{N}(w) = \lbrace z \in \mathbb{R}^d: z'A'W_j = 0, ~\forall j = 1,\ldots,N \rbrace$. Therefore, 
\begin{equation}
\mathcal{R}(w) = \mathcal{N}(w)^\perp = \linspan{ A'W_1,\ldots, A'W_N }.
\end{equation}
In turn, because $\lbrace w_0,\ldots,w_{N-1} \rbrace = \lbrace W_1,\ldots,W_N \rbrace$, $T$ is at most $N-1$.

By \cref{corollary: criteria system solution}, $Ax_{T+1} = b$ if and only if $P_{\mathcal{V}(w)} x_0 = P_{\mathcal{V}(w)}x^*$ where $x^*$ satisfies $Ax^* = b$. Now, given that $\mathcal{R}(w) + \mathcal{V}(w) = \mathrm{row}(A)$ and $\mathcal{R}(w) = \linspan{ A'W_1,\ldots, A'W_N }$, if $\linspan{A'W_1,\ldots,A'W_N} = \mathrm{row}(A)$ then $\mathcal{V}(w) = \lbrace 0 \rbrace$. Therefore, $Ax_{T+1} = b$ for any initialization. The final claim is straightforward.
\end{proof}

The second sampling scheme, independent and identically distributed sampling, is exemplified by randomly sampling equations from the system with uniform discrete probability. However, we do not need to limit ourselves to sampling from a finite population of elements. As the next result shows, we can do much more.

\begin{proposition} \label{theorem: iid sampling}
Suppose that $w, w_0,w_1,\ldots$ are independent, identically distributed random variables. There exists a $\pi \in (0,1)$ such that
\begin{equation} \label{theorem-cond:min-prob}
\mathop{\inf _{v \in \mathcal{R}(w)}}_{ \norm{v}_2 = 1} \Prb{ v'A'w \neq 0} \geq \pi.
\end{equation}
Moreover, $T < \infty$ and $\Prb{ T =  k } \leq (k - r) ^{r-1} (1- \pi) ^ {k - r}$ where $r = \dim( \mathcal{R}(w))$ and $k \geq r$.
\end{proposition}
\begin{proof}
First, we  show that there exists $\pi>0$ such that for any nontrivial, proper subspace $V\subsetneq\mathcal{R}(w)$,
$
\PP[A'w\not\in V] \geq \pi$, which implies \cref{theorem-cond:min-prob} when we take $V$ to be the relative orthogonal compliment to the span of a unit vector $v \in \mathcal{R}(w)$. Suppose there is no such $\pi$. Then, for every $p \in (0,1)$, there is a nontrivial subspace $V \subsetneq \mathcal{R}(w)$ such that $\Prb{ A'w \in V} \geq 1 - p$. Let $r$ be the smallest integer between $0$ and $\dim(\mathcal{R}(w))$ such that
\begin{equation}
\mathop{\sup _{V\subsetneq \mathcal{R}(w)}}_{\dim[V] = r} \PP[A'w \in V] = 1.
\end{equation}

For  $\epsilon>0$, let $V_1\subsetneq\mathcal{R}(w)$ be an $r$-dimension subspace with $\PP[A'w \in V_1] \geq 1 - \epsilon/2$. Note, by \cref{lemma: characterize R_w}, $\PP[A'w \in V_1] < 1$. Therefore, let $V_2 \subsetneq \mathcal{R}(w)$ be an $r$-dimensional subspace with $\PP[A'w \in V_2] > \PP[A'w \in V_1] \geq 1 - \epsilon/2$. Given that $V_1$ and $V_2$ are distinct and the inclusion-exclusion principle, 
\begin{equation}
\PP[A'w\in V_1\cap V_2] \geq \PP[A'w\in V_1]  + \PP[A'w\in V_2]  -1\geq 1-\epsilon.
\end{equation}
However, this is contradicts the minimality of $r$ since $\epsilon >0$ is arbitrary and $\dim(V_1 \cap V_2) < r$. Thus, we conclude that such a $\pi$ exists.

It follow from \cref{theorem-cond:min-prob} that for any k,
\begin{equation}
\PP\left[\dim(\linspan{A'w_0,\ldots,A'w_k}) > \dim(\linspan{A'w_0,\ldots,A'w_{k-1}}\right) ] \geq \pi.
\end{equation}
Therefore, we can bound $\PP[T=k]$ by a negative binomial distribution. In particular,
\begin{equation}
\PP[T=k] \leq \binom{k-1}{ r -1}\, (1- \pi) ^ {k - r} \leq (k - r) ^{r-1} (1- \pi) ^ {k - r}.
\end{equation}
\end{proof}

In light of the two preceding results, we may be convinced that there is a gap between the convergence properties between random permutation sampling and the independent and identically distributed sampling. However, by modifying the structure of the rank-one RPM, we can find more intermediate cases. The next result demonstrates this behavior with a somewhat contrived example, and we will leave more complex cases to future work.

\begin{theorem}
Suppose $w , w_0, w_1,\ldots$ are i.i.d. random variables such that the entries of $A'w$ are independent, identically distributed subgaussian random variables with mean zero and unit variance. 
Then, there exists a $\pi \in (0,1)$ depending only on the distribution of the entries of $A'w$ such that $\Prb{T = k} \geq 1 - \pi^k$ for $k \geq d$.
\end{theorem}
\begin{proof} Let $H_k$ denote a $k \times d$ ($k \geq d$) random matrix whose entries are independent and identically distributed subgaussian random variables with zero mean and unit variance. As a consequence of \cite[Theorem 1.1]{rudelson2009}, there exists a $\pi$ that depends on the distribution of the entries such that for all $k \geq d$, $\Prb{ \sigma_{\min}(H_k) > 0 } \geq 1- \pi^k$. At iteration $k$, let $N_k$ denote the matrix whose rows are given by $w_0,w_1,\ldots$. Then, by hypothesis, $N_kA$ has entries that are independent, identically distributed subgaussian random with zero mean and unit variance. Therefore, there exists a $\pi \in (0,1)$ depending only on the distribution of the entries in $A'w$ such that $\Prb{ T = k} = \Prb{ \sigma_{\min}(N_kA) > 0} \geq 1 - \pi^k$ for $k \geq d$.
\end{proof}

\section{Convergence Theory for Base Methods} \label{section:no-memory}
In the previous section, we proved convergence for the complete orthogonalization method (i.e., \cref{alg: rank-one RPM full mem}) and explored some specific sampling patterns. Here, we will consider the extreme opposite of the complete orthogonalization method: the ``base'' randomized iterative approach (e.g., Randomized Kaczmarz). That is, we consider when $V_k$ is a rank one matrix of one of two general classes.

In the first class, we consider \cref{alg: rank-one RPM low mem} in the case $m = 0$. In this case, \cref{eqn: rank-one update-param} supplies the simplified iteration scheme,
\begin{equation} \label{eqn:no-memory-iteration}
x_{k+1} = x_k + \frac{A'w_k w_k'(b - Ax_k)}{\norm{A'w_k}_2^2},
\end{equation} 
which encompasses randomized Kaczmarz, when $w_k$ is a random draw from the standard basis vectors in $\mathbb{R}^n$, as shown in \cref{subsection:overview}. 

Unfortunately, \cref{eqn:no-memory-iteration} would not include randomized Gauss-Seidel. This motivates the second class, which has the closely related iteration
\begin{equation} \label{eqn:no-memory-column-iteration}
x_{k+1} = x_k + \frac{w_k w_k'A'(b - A x_k)}{ \norm{Aw_k}_2^2 }.
\end{equation}
In this class, we recover randomized Gauss-Seidel if we choose $w_k$ randomly from the standard basis vectors in $\mathbb{R}^d$, as shown in \cref{subsection:overview}.

While these two classes are distinct, we will see that their analysis is nearly identical and is intimately related to the analysis of the complete orthogonalization method. Our analysis offers two highlights: (1) we can prove convergence with probability one for arbitrary sampling schemes---only the i.i.d. case is considered in \cite{zouzias2013,gower2015,richtarik2017}; and (2) we can provide rates of convergence with probability one which complements the mean-squared-error results of \cite{zouzias2013,gower2015,richtarik2017}. Our main approach is an extension of Meany's inequality (see \cref{subsection:meany}) combined with stopping time arguments, as derived in \cref{subsection:no-mem-convergence,subsection:no-mem-column}. We then explore some common, non-adaptive sampling patterns in \cref{subsection:no-mem-sampling-patterns}. To conclude, we develop a general framework for the analysis of adaptive sampling schemes, and provide concrete examples from the literature (see \cref{subsection:adaptive-sampling}).

\subsection{An Extension of Meany's Inequality} \label{subsection:meany}

Here, we will derive an extension of Meany's Inequality \cite{meany1969}, which, under a different extension, has recently been used to study the convergence rate of row-action solvers including the a block-variant of the Kaczmarz method \cite{bai2013}. We begin by stating a geometric lemma derived by \cite{meany1969}, and follow it with the extension, which closely follows Meany's original proof with several modifications.

\begin{lemma}[\cite{meany1969}] \label{lemma:meany-lemma}
Let $f_1,\ldots,f_k \in \mathbb{R}^n$ with $k \leq n$. Write $f_k = f^S + f^N$ where $f^S$ belongs to the space $S$ spanned by $f_1,\ldots,f_{k-1}$ and $f^N$ is perpendicular to $S$. Let $\bar{F}$ be the matrix whose columns are $f_1,\ldots,f_{k-1}$, and let $F$ be the matrix whose columns are $f_1,\ldots,f_k$. Then,
\begin{equation}
\det( F'F) = \norm{ f^N}_2^2 \det( \bar{F}'\bar{F}).
\end{equation}
\end{lemma}

\begin{theorem} \label{theorem:meany-no-mem}
Let $v_1,\ldots,v_k$ be unit vectors in $\mathbb{R}^n$ for some $k \in \mathbb{N}$. Let $S = \linspan{v_1,\ldots,v_k}$. Let $\mathcal{F}$ denote all matrices $F$ where the columns of $F$ are the vectors $\lbrace f_1,\ldots,f_r \rbrace \subset \lbrace v_1,\ldots,v_k \rbrace$ that are a maximal linearly independent subset. Then
\begin{equation}
\sup_{ y \in S, \norm{y}_2 = 1} \norm{Q y}_2 \leq \sqrt{1 - \min_{ F \in \mathcal{F}} \det(F'F)},
\end{equation}
where
\begin{equation}
Q = (I - v_k v_k')(I - v_{k-1}v_{k-1}')\cdots (I - v_1 v_1').
\end{equation}
\end{theorem}
\begin{proof}
The proof proceeds by induction. For the case $k=1$, both sides of the inequality are zero and so the result holds. Now suppose that the result holds for $k = j-1$. To prove the case $k=j$, we need the following additional notation.

Let $\bar{S} = \linspan{ v_1,\ldots,v_{j-1}} $; let $\lbrace f_1,\ldots,f_{\bar{r}} \rbrace$ denote a maximal linearly independent subset of the unit vectors $\lbrace v_1,\ldots,v_{j-1} \rbrace$ that achieve the minimum determinant; let $\bar{F}$ be the matrix whose columns are $f_1,\ldots,f_{\bar{r}}$; and let
\begin{equation}
\bar{Q} = (I - v_{j-1} v_{j-1}')( I - v_{j-2} v_{j-2}') \cdots (I - v_1 v_1').
\end{equation}
For a unit vector $y \in S$, let $y^{\bar{S}}$ denote the component of $y$ in $\bar{S}$, and let $y^N$ denote the component of $y$ orthogonal to $\bar{S}$. Moreover, let $z = \bar{Q} y^{\bar{S}}$. Then, by the induction hypothesis,
\begin{equation} \label{proof-meany:induction-hyp}
\norm{z}_2 = \norm{ \bar{Q} y^{\bar{S}}}_2 \leq \norm{y^{\bar{S}}}_2 \sqrt{1 - \det(\bar{F}'\bar{F})}.
\end{equation} 
Similarly, write $v_j = v^{\bar{S}} + v^N$ where $v^{\bar{S}} \in \bar{S}$ and $v^N$ is perpendicular to $\bar{S}$. 

\underline{Case A:} Suppose that $S = \bar{S}$. Then $y = y^{\bar{S}}$. Moreover, since $\bar{F} \in \mathcal{F}$,
\begin{equation}
\norm{ Q y}_2  \leq \norm{ \bar{Q} y}_2 \leq \norm{ y }_2 \sqrt{ 1 - \det(\bar {F} ' \bar{F}) } \leq \norm{y}_2 \sqrt{ 1 - \min_{F \in \mathcal{F}} \det( F'F) }.
\end{equation} 
Thus, the result holds when $S = \bar{S}$. 

\underline{Case B:} Suppose that $S \supsetneq \bar{S}$. Then,
\begin{align}
\norm{Qy}_2^2 &= \norm{ (I - v_j v_j')(z + y^N) }_2^2 = (z + y^N)' (I - v_j v_j') (z + y^N) \\
			  &= \norm{z}_2^2 + \norm{y^N}^2 + \underbrace{2 z' y^N}_{0} - (\underbrace{z' v_j}_{z'v^{\bar{S}}})^2 - 2 \underbrace{z'v_j}_{ z'v^{\bar{S}}} \underbrace{v_j'y^N}_{(v^N)'y^N} - (\underbrace{v_j' y^N}_{{(v^N)'y^N}})^2 \\
			  &= \norm{z}_2^2 + \norm{y^N}^2 - (z' v^{\bar{S}})^2 - 2 z'v^{ \bar{S} } \norm{v^N}_2 \norm{y^N}_2 - \norm{v^N}_2^2 \norm{y^N}_2^2,
\end{align}
where we have made use of $v^N$ and $y^N$ are colinear, implying that their inner product is equal to the product of their norms. Finally, since $-2 z ' v^{\bar{S}} \leq 2 |z' v^{\bar{S}}|$, 
\begin{equation} \label{proof-meany:main-inequality}
\norm{Qy}_2^2 \leq  \norm{z}_2^2 + \norm{y^N}_2^2 - \left( \left| z' v^{\bar{S}} \right| - \norm{v^N}_2 \norm{y^N}_2  \right)^2.
\end{equation}

\underline{Case B(1):} Suppose that $\norm{v^N}_2 \leq \norm{y^{\bar{S}}}_2$. Then,
\begin{align}
\norm{Qy}_2^2 &\leq \norm{z}_2^2 + \norm{y^N}_2^2 - \left( \left| z' v^{\bar{S}} \right| - \norm{v^N}_2 \norm{y^N}_2  \right)^2 \tag{by \cref{proof-meany:main-inequality}} \\
& \leq \norm{z}_2^2 + \norm{y^N}_2^2 \nonumber \\
& \leq \norm{y^{\bar{S}}}_2^2(1 - \det(\bar{F}'\bar{F})) + \norm{y^N}_2^2 \tag{by \cref{proof-meany:induction-hyp}} \\
& = \norm{y}_2^2 - \norm{y^{\bar{S}}}_2^2 \det(\bar{F}'\bar{F}) \nonumber \\
& \leq 1 - \norm{v^N}_2^2 \det(\bar{F}'\bar{F}) \nonumber \tag{ $\norm{y}_2 = 1$ and $\norm{v^N}_2 \leq \norm{y^{\bar{S}}}_2$} \\
& \leq 1 - \min_{F \in \mathcal{F}}\det(F'F),
\end{align}
where, in the last line, we use \cref{lemma:meany-lemma} and, since $S \neq \bar{S}$, $f_{\bar{r}+1} = v_j$, which, in turn, implies $f^N = v^N$.

\underline{Case B(2):} Suppose that $\norm{v^N}_2 > \norm{y^{\bar{S}}}_2$. Since $\norm{v_j}_2 = \norm{y}_2 = 1$, then $\norm{v^{\bar{S}}}_2 \leq \norm{ y^N}_2$. Using these inequalities and \cref{proof-meany:induction-hyp},
\begin{equation}
\norm{y^N}_2 \norm{v^N}_2 \geq \norm{v^{\bar{S}}}_2 \norm{y^{\bar{S}}}_2 \geq \norm{ v^{\bar{S}}}_2 \norm{ z}_2 \geq | z' v^{\bar{S}}|.
\end{equation}
Therefore, 
\begin{equation}
\norm{y^N}_2 \norm{v^N}_2 - | z' v^{\bar{S}}| \geq \norm{y^N}_2 \norm{v^N}_2 - \norm{z}_2 \norm{v^{\bar{S}}}_2 \geq 0.
\end{equation}
Applying this relationship to \cref{proof-meany:main-inequality},
\begin{align*}
\norm{Q y}_2^2 &\leq \norm{z}_2^2 + \norm{y^N}_2^2 - \left(  \norm{y^N}_2 \norm{v^N}_2 - \norm{z}_2 \norm{v^{\bar{S}}}_2 \right)^2 \\
			   &= \norm{z}_2^2 \norm{v^N}_2^2 + \norm{y^N}_2^2 \norm{v^{\bar{S}}}_2^2 + 2 \norm{v^{\bar{S}}}_2 \norm{z}_2 \norm{y^N}_2 \norm{v^N}_2  \\
			   &= \left( \norm{z}_2 \norm{v^N}_2 + \norm{y^N}_2 \norm{v^{\bar{S}}}_2 \right)^2 \\
			   &\leq \left( \sqrt{ 1 - \det( \bar{F}'\bar{F})} \norm{v^N}_2 \norm{y^{\bar{S}}}_2 + \norm{y^N}_2 \norm{v^{\bar{S}}}_2 \right)^2 \tag{by \cref{proof-meany:induction-hyp}} \\
			   &\leq \left( \norm{y^{\bar{S}}}_2^2 + \norm{y^N}_2^2 \right) \left( \norm{v^N}_2^2 ( 1 - \det( \bar{F}'\bar{F})) + \norm{v^{\bar{S}}}_2^2 \right) \tag{by Cauchy-Schwarz} \\
			   &= 1 - \norm{v^N}_2^2 \det( \bar{F}' \bar{F} ) \\
			   &= 1 - \min_{F \in \mathcal{F}} \det( F' F),
\end{align*}
where, in the last line, we use \cref{lemma:meany-lemma} and, since $S \neq \bar{S}$, $f_{\bar{r}+1} = v_j$, which, in turn, implies $f^N = v^N$.

Therefore, from Cases A, B(1) and B(2), we conclude that the result holds.
\end{proof}

\subsection{Main Convergence Result for Row-Action Methods} \label{subsection:no-mem-convergence}

Recall that $w \in \mathbb{R}^n$ is a random variable and $\lbrace w_\ell : \ell+1 \in \mathbb{N} \rbrace$ is a sequence of random variables taking value in $\mathbb{R}^n$ chosen such that $A'w_\ell \in \mathcal{R}(w)$.\footnote{Again, we can avoid this requirement and consider set inclusions below. However, this generalization will require additional, cumbersome notation and there is no practical reason for considering this case.} We will now define a sequence of stopping times $\lbrace \tau_\ell : \ell+1 \in \mathbb{N} \rbrace$ where $\tau_0 = 0$,
\begin{equation} \label{eqn:stop-iter-0}
\tau_1 = \min\lbrace k \geq 0 : \linspan{A'w_0,\ldots,A'w_k} = \mathcal{R}(w) \rbrace,
\end{equation}
and, if $\tau_{\ell-1} < \infty$, we define
\begin{equation} \label{eqn:stop-iter-arbitrary}
\tau_\ell = \min\lbrace k > \tau_{\ell-1} : \linspan{ A'w_{\tau_{\ell-1}+1},\ldots,A'w_{k}} = \mathcal{R}(w) \rbrace,
\end{equation}
else $\tau_\ell = \infty$. As an aside, it is worthwhile to note the commonalities between the definition of $\lbrace \tau_\ell \rbrace$ and the stopping time $T$ from \cref{eqn: stopping time}.

Moreover, whenever the stopping times are finite, we will define the collection, $\mathcal{F}_\ell$, for $\ell \in \mathbb{N}$, that contains all matrices $F$ whose columns are maximal linearly independent subsets of 
\begin{equation}
\left\lbrace \frac{A'w_{\tau_{\ell-1}+1}}{\norm{A'w_{\tau_{\ell-1}+1}}_2},\ldots,\frac{A'w_{\tau_\ell}}{\norm{A'w_{\tau_\ell}}_2} \right\rbrace.
\end{equation}
Moreover, define
\begin{equation} \label{eqn:no-mem-rate}
\gamma_\ell = 1 - \min_{F \in \mathcal{F}_l} \det( F' F).
\end{equation}
Note, it follows by Hadamard's inequality that $\gamma_\ell \in [0,1)$.

\begin{theorem} \label{theorem:no-mem-convergence}
Suppose $Ax=b$ admits a solution $x^*$ (not necessarily unique). Let $w$ be a random variable valued in $\mathbb{R}^n$, and let $\mathcal{R}(w),$ $\mathcal{N}(w)$ and $\mathcal{V}(w)$ be defined as above (see \cref{eqn: row span}). Moreover, let $\lbrace w_\ell : \ell+1 \in \mathbb{N} \rbrace $ be random variables such that $\Prb{ A'w_\ell \in \mathcal{R}(w)} = 1$ for all $\ell+1 \in \mathbb{N}$. Let $x_0 \in \mathbb{R}^d$ be arbitrary and let $\lbrace x_k : k \in \mathbb{N} \rbrace$ be defined as in \cref{eqn:no-memory-iteration}. Then, for any $\ell$, on the event $\lbrace \tau_\ell < \infty \rbrace$,
\begin{equation} \label{theorem-eqn:no-mem-stop-rate}
\norm{x_{\tau_{\ell}+1} - x^* - P_{\mathcal{N}(w)}(x_0 - x^*)}_2^2 \leq \left( \prod_{j=1}^\ell \gamma_j  \right) \norm{ P_{\mathcal{R}(w)}(x_0 - x^*) }_2^2,
\end{equation}
where $\gamma_j$ are defined in \cref{eqn:no-mem-rate} and $\gamma_j \in [0,1)$. Therefore, for any $k$,
\begin{equation}
\norm{ x_k - x^* - P_{\mathcal{N}(w)}(x_0 - x^*) }_2^2 \leq \left( \prod_{j=1}^{L(k)} \gamma_j \right) \norm{ P_{\mathcal{R}(w)}(x_0 - x^*) }_2^2,
\end{equation}
where $L(k) = \max\lbrace \ell : k \geq \tau_\ell + 1 \rbrace$; and where we are on the event $\lbrace \tau_{L(k)} < \infty \rbrace$.
\end{theorem}
\begin{proof}
From the basic iteration stated in \cref{eqn:no-memory-iteration}, we have
\begin{equation} \label{proof-eqn:no-mem-iteration}
x_{k+1} - x^* = x_k - x^* - \frac{A' w_k w_k' A}{\norm{A'w_k}_2^2} (x_k - x^*) = \left( I - \frac{A'w_k w_k'A}{\norm{A'w_k}_2^2} \right) (x_k - x^*).
\end{equation}
Iterating on this relationship, we conclude
\begin{equation}
x_{k+1} - x^* = \left( I - \frac{A'w_k w_k'A}{\norm{A'w_k}_2^2} \right) \cdots \left( I - \frac{A'w_0 w_0'A}{\norm{A'w_0}_2^2} \right) (x_0 - x^*).
\end{equation} 
Moreover, by assumption, $A'w_\ell \in \mathcal{R}(w)$ with probability one, which implies that $A'w_\ell \perp \mathcal{N}(w)$. Therefore,
\begin{equation} \label{proof-eqn:decomposed-iteration}
x_{k+1} - x^* = P_{\mathcal{N}(w)}(x_0 - x^*) + \left( I - \frac{A'w_k w_k'A}{\norm{A'w_k}_2^2} \right) \cdots \left( I - \frac{A'w_0 w_0'A}{\norm{A'w_0}_2^2} \right) P_{\mathcal{R}(w)}(x_0 - x^*),
\end{equation}
and $P_{\mathcal{N}(w)}(x_{k} - x^*) = P_{\mathcal{N}(w)}(x_0 - x^*)$.

Note, when $\tau_1$ is finite, then the span of $\lbrace A'w_0,\ldots,A'w_{\tau_1} \rbrace$ is $\mathcal{R}(w)$. Therefore, on the event $\tau_1 < \infty$, \cref{theorem:meany-no-mem} implies that
\begin{equation}
\norm{x_{\tau_1 + 1} - x^* - P_{\mathcal{N}(w)}(x_0 - x^*)}_2^2 \leq \gamma_1 \norm{P_{\mathcal{R}(w)}(x_0 - x^*)}_2^2.
\end{equation}
We now proceed by induction. Suppose \cref{theorem-eqn:no-mem-stop-rate} holds for some $\ell \in \mathbb{N}$. Using \cref{proof-eqn:decomposed-iteration}, for $k > \tau_\ell$,
\begin{equation}
\begin{aligned}
&x_k - x^* - P_{\mathcal{N}(w)}(x_0 - x^*) \\
&\quad = \left( I - \frac{A'w_k w_k'A}{\norm{A'w_k}_2^2} \right) \cdots \left( I - \frac{A'w_{\tau_\ell+1} w_{\tau_\ell+1}'A}{\norm{A'w_{\tau_\ell+1}}_2^2} \right) P_{\mathcal{R}(w)}(x_{\tau_\ell+1} - x^*).
\end{aligned}
\end{equation}
Now, when $k = \tau_{\ell+1} +1$, the conditions of \cref{theorem:meany-no-mem} are satisfied. Therefore,
\begin{equation}
\begin{aligned}
\norm{x_{\tau_{\ell+1} + 1} - x^* - P_{\mathcal{N}(w)}(x_0 - x^*)}_2^2 &\leq \gamma_{\ell+1} \norm{ P_{\mathcal{R}(w)}(x_{\tau_\ell +1} - x^*) }_2^2 \\ 
&= \gamma_{\ell+1} \norm{x_{\tau_\ell + 1} - x^* - P_{\mathcal{N}(w)}(x_0 - x^*)}_2^2.
\end{aligned}
\end{equation}
By applying the induction hypothesis, we conclude that \cref{theorem-eqn:no-mem-stop-rate} holds on the event $ \lbrace \tau_{\ell+1} < \infty \rbrace$.

Now, for an orthogonal projection matrix, $I- vv'$, $\norm{I - vv'}_2 = 1$. The bound on $x_k - x^* - P_{\mathcal{N}}(x_0 - x^*)$ follows by applying this fact and the definition of $L(k)$.
\end{proof}

As an analogue of \cref{corollary: criteria system solution}, we have the following characterization of whether $\lim_{k \to \infty} x_k$ solves the system $Ax=b$.

\begin{corollary}
Under the setting of \cref{theorem:no-mem-convergence}, on the events $\bigcap_{\ell=0}^\infty \lbrace \tau_\ell < \infty \rbrace$ and $\lbrace \lim_{\ell \to \infty} \prod_{j=0}^\ell \gamma_j = 0 \rbrace$, $\lim_{k \to \infty} Ax_k = b$ if and only if $P_{\mathcal{V}(w)} x_0 = P_{\mathcal{V}(w)} x^*$.
\end{corollary}
\begin{proof}
By \cref{theorem:no-mem-convergence}, and on the events $\bigcap_{\ell=0}^\infty \lbrace \tau_\ell < \infty \rbrace$ and $\lbrace \lim_{\ell \to \infty} \prod_{j=1}^\ell \gamma_j = 0 \rbrace$,
\begin{equation}
\lim_{k \to \infty} x_k = x^* + P_{\mathcal{N}(w)}(x_0 - x^*) = x^* + P_{\ker(A)} (x_0 - x^*) + P_{\mathcal{V}(w)} (x_0 - x^*).
\end{equation}
Therefore, $\lim_{k \to \infty} Ax_k = b + AP_{\mathcal{V}(w)} (x_0 - x^*)$, which implies $\lim_{k \to \infty} Ax_k = b$ if and only if $AP_{\mathcal{V}(w)} x_0 = AP_{\mathcal{V}(w)} x^*$. Clearly, if $P_{\mathcal{V}(w)} x_0 = P_{\mathcal{V}(w)}x^*$, then $AP_{\mathcal{V}(w)}x_0 = AP_{\mathcal{V}(w)}x^*$. Now, since $\mathcal{V}(w) \subset \mathrm{row}(A)$, if $AP_{\mathcal{V}(w)} x_0 = AP_{\mathcal{V}(w)} x^*$, then $P_{\mathcal{V}(w)}x_0 = P_{\mathcal{V}(w)} x^*$ follows from \cref{proof-eqn:V-in-row-A}.
\end{proof}

\subsection{Main Convergence Result for Column-Action Methods} \label{subsection:no-mem-column}

For the family of methods specified by \cref{eqn:no-memory-column-iteration}, we will follow an almost identical proof except on the residual rather than the error. Specifically, if we let $r_k = Ax_k - b$, then \cref{eqn:no-memory-column-iteration} implies
\begin{equation} \label{eqn:no-mem-col-residual-iter}
r_{k+1} = Ax_{k+1} - b = Ax_k - b - \frac{A w_k w_k' A'}{\norm{A w_k}_2^2} ( Ax_k - b) = \left( I - \frac{A w_k w_k' A'}{\norm{A w_k}_2^2} \right) r_k.
\end{equation}
Thus, we will see two changes in the proof. First, we will see that see that $r_k$ for column-action methods will take the place of $x_k - x^*$ for row-action methods. Second, we already see that $Aw_k$ in \cref{eqn:no-mem-col-residual-iter} has taken the place of $A'w_k$ in \cref{proof-eqn:no-mem-iteration}. Owing to this latter issue, we will need to specify analogues of $\mathcal{R}(w)$, $\mathcal{N}(w)$ and $\mathcal{V}(w)$. 

Let $w \in \mathbb{R}^d$ be a random variable, and let
\begin{equation} \label{eqn:left-col-null}
\mathcal{L}(w) = \mathrm{span}\left[ z \in \mathbb{R}^n : \Prb{ z' A w = 0 } = 1 \right]~\mathrm{and}~ \mathcal{C}(w) = \mathcal{L}(w)^\perp.
\end{equation}
Just as $\mathcal{N}(w)$ generalized the null space of $A$ under the action of an $n$-dimensional random variable from the left, we see that $\mathcal{L}(w)$ is a generalization of the left null space of $A$ under the action of a $d$-dimensional random variable from the right. Analogously, just as $\mathcal{R}(w)$ restricted the row space of $A$ under the action of an $n$-dimensional random variable from the left, we see that $\mathcal{C}(w)$ is a restriction of the column space of $A$ under the action of a $d$-dimensional random variable from the right. Finally, we let $\mathcal{E}(w)$ denote the subspace that is orthogonal to $\mathcal{C}(w)$ such that $\mathcal{E}(w) \oplus \mathcal{C}(w)$ is the column space of $A$. 

With these new definitions, we may proceed just as we do in \cref{subsection:no-mem-convergence}. For a random variable $w \in \mathbb{R}^d$, let $\lbrace w_\ell : \ell+1 \in \mathbb{N} \rbrace$ be a sequence of random variables in $\mathbb{R}^d$ such that $Aw_\ell \in \mathcal{C}(w)$. We will now define a sequence of stopping times $\lbrace \tau_\ell : \ell + 1\in \mathbb{N} \rbrace$ where $\tau_0 = 0$,
\begin{equation}
\tau_1 = \min \lbrace k \geq 0 : \mathrm{span}\left[ Aw_0,\ldots,Aw_k \right]  = \mathcal{C}(w) \rbrace,
\end{equation}
and, if $\tau_{\ell - 1} < \infty$, we define 
\begin{equation}
\tau_\ell = \min \lbrace k > \tau_{\ell-1} : \mathrm{span} \left[ A w_{\tau_{\ell -1} + 1},\ldots, A w_k \right] = \mathcal{C}(w) \rbrace,
\end{equation}
else $\tau_\ell = \infty$. 

Moreover, whenever the stopping times are finite, we will define a collection, $\mathcal{F}_{\ell},$ for $\ell \in \mathbb{N}$, that contains all matrices $F$ whose columns are maximal linearly independent subsets of 
\begin{equation}
\left\lbrace \frac{Aw_{\tau_{\ell-1}+1}}{\norm{Aw_{\tau_{\ell-1}+1}}_2},\ldots,\frac{Aw_{\tau_\ell}}{\norm{Aw_{\tau_\ell}}_2} \right\rbrace.
\end{equation}
We can then define $\gamma_\ell$ just as we do in \cref{eqn:no-mem-rate}. For completeness, we will define it again here so that we reference the appropriate definitions. Define
\begin{equation} \label{eqn:no-mem-col-rate}
\gamma_\ell = 1 - \min_{ F \in \mathcal{F}_\ell} \det( F' F).
\end{equation}

\begin{theorem} \label{theorem:no-mem-column-convergence} 
Suppose $Ax = b$ admits a solution $x^*$ (not necessarily unique). Let $w$ be a random variable valued in $\mathbb{R}^d$, and let $\mathcal{C}(w), \mathcal{L}(w)$ and $\mathcal{E}(w)$ be defined as above (see \cref{eqn:left-col-null}). Moreover, let $\lbrace w_\ell : \ell + 1 \in \mathbb{N} \rbrace$ be random variables such that $\Prb{ A w_\ell \in \mathcal{C}(w)} = 1$ for all $\ell +1 \in \mathbb{N}$. Let $x_0 \in \mathbb{R}^d$ be arbitrary, let $\lbrace x_k : k \in \mathbb{N} \rbrace$ be defined as in \cref{eqn:no-memory-column-iteration}, and define $r_k = Ax_k - b$ for $k + 1 \in \mathbb{N}$. Then, for any $\ell$, on the event $\lbrace \tau_{\ell} < \infty \rbrace$, 
\begin{equation} \label{theorem-eqn:no-mem-col-rate}
\norm{ r_{\tau_{\ell} +1} - P_{\mathcal{L}(w)} r_0 }_2^2 \leq \left( \prod_{j=1}^\ell \gamma_j \right) \norm{ P_{\mathcal{C}(w)} r_0 }_2^2,
\end{equation}
where $\gamma_j$ are defined in \cref{eqn:no-mem-col-rate} and $\gamma_j \in [0,1)$. Therefore, for any $k$,
\begin{equation}
\norm{ r_{k} - P_{\mathcal{L}(w)} r_0 }_2^2 \leq \left( \prod_{j=1}^{L(k)} \gamma_j \right) \norm{ P_{\mathcal{C}(w)} r_0 }_2^2,
\end{equation}
where $L(k) = \max\lbrace \ell: k \geq \tau_{\ell}+1 \rbrace$; and where we are on the event $\lbrace \tau_{L(k)} < \infty \rbrace$.
\end{theorem}
\begin{proof}
Iterating on \cref{eqn:no-mem-col-residual-iter}, we conclude 
\begin{equation}
r_{k+1} = \left( I - \frac{A w_k w_k' A'}{\norm{A w_k}_2^2} \right) \cdots \left( I - \frac{A w_0 w_0' A'}{\norm{A w_0}_2^2} \right) r_0.
\end{equation}
Moreover, by assumption, $A w_\ell \in \mathcal{C}(w)$ with probability one, which implies $A w_\ell \perp \mathcal{L}(w)$. Therefore,
\begin{equation} \label{proof-eqn:no-mem-col-decomp}
r_k = P_{\mathcal{L}(w)} r_0 + \left( I - \frac{A w_k w_k' A'}{\norm{A w_k}_2^2} \right) \cdots \left( I - \frac{A w_0 w_0' A'}{\norm{A w_0}_2^2} \right) P_{\mathcal{C}(w)} r_0,
\end{equation}
and $P_{\mathcal{L}(w)}r_k = P_{\mathcal{L}(w)} r_0.$

Note, when $\tau_1$ is finite, then the span of $\lbrace Aw_0,\ldots, Aw_{\tau_1} \rbrace$ is $\mathcal{C}(w)$. Therefore, on the event $\tau_1 < \infty$, \cref{theorem:meany-no-mem} implies that
\begin{equation}
\norm{ r_{\tau_1 + 1} - P_{\mathcal{L}(w)} r_0 }_2^2 \leq \gamma_1 \norm{ P_{\mathcal{C}(w)} r_0 }_2^2.
\end{equation}
We now proceed by induction. Suppose \cref{theorem-eqn:no-mem-col-rate} holds for some $\ell \in \mathbb{N}$. Using \cref{proof-eqn:no-mem-col-decomp}, for $k > \tau_\ell$, 
\begin{equation}
r_k - P_{\mathcal{L}(w)}r_0 = \left( I - \frac{Aw_k w_k'A'}{\norm{Aw_k}_2^2} \right) \cdots \left( I - \frac{A w_{\tau_\ell + 1}w_{\tau_\ell + 1}'A'}{\norm{ A w_{\tau_\ell + 1} }_2^2 } \right) P_{\mathcal{C}(w)} r_{\tau_\ell + 1}.
\end{equation}
Now, when $k = \tau_{\ell + 1} + 1$, the conditions of \cref{theorem:meany-no-mem} are satisfied. Therefore,
\begin{equation}
\begin{aligned}
\norm{ r_{\tau_{\ell+1} + 1} - P_{\mathcal{L}(w)} r_0 }_2^2 
&\leq \gamma_{\ell+1} \norm{ P_{\mathcal{C}(w)} r_{\tau_\ell+1}}_2^2 \\
&= \gamma_{\ell+1} \norm{ r_{\tau_\ell + 1} - P_{\mathcal{L}(w)} r_0 }_2^2.
\end{aligned}
\end{equation}
By applying the induction hypothesis, we conclude that \cref{theorem-eqn:no-mem-col-rate} holds on the event $\lbrace \tau_{\ell+1} < \infty \rbrace$. The second part of the result follows readily.
\end{proof}

We have the following characterization of whether $\lim_{k \to \infty} x_k$ solves the system $Ax = b$. 

\begin{corollary}
Under the setting of \cref{theorem:no-mem-column-convergence}, on the events $\lbrace \bigcap_{\ell = 0}^\infty \tau_\ell < \infty \rbrace$, and $\lbrace \lim_{\ell \to\infty } \prod_{j=0}^\ell \gamma_j = 0 \rbrace$, $\lim_{k \to \infty} Ax_k = b$ if and only if $P_{\mathcal{E}(w)} r_0 = 0$.
\end{corollary}
\begin{proof}
On the events $\lbrace \bigcap_{\ell = 0}^\infty \tau_\ell < \infty \rbrace$ and $\lbrace \lim_{\ell \to\infty } \prod_{j=0}^\ell \gamma_j = 0 \rbrace$, \cref{theorem:no-mem-column-convergence} implies
\begin{equation}
\lim_{k \to \infty} r_k = P_{\mathcal{L}(w)} r_0.
\end{equation}
It straightforwardly follows that $\lim_{k \to \infty} Ax_k = b$ if and only if $P_{\mathcal{L}(w)} r_0 = 0$. 

Moreover, by construction of $\mathcal{L}(w)$, we have that $\mathcal{L}(w) = \mathcal{E}(w) \oplus \ker(A')$. Thus,
\begin{equation}
P_{\mathcal{L}(w)} r_0 = P_{\mathcal{E}(w)} r_0 + P_{\ker(A')} r_0.
\end{equation}
Since the left null space of $A$ is orthogonal to the column space of $A$, and  $r_0$ is in the column space of $A$ because $Ax=b$ is consistent, we have that $P_{\mathcal{L}(w)} r_0 = P_{\mathcal{E}(w)} r_0$.
\end{proof}

\subsection{Common, Non-Adaptive Sampling Patterns} \label{subsection:no-mem-sampling-patterns}

Just as for \cref{theorem: terminal iteration characterization}, \cref{theorem:no-mem-convergence,theorem:no-mem-column-convergence} are general results that characterizes convergence for \textit{any} sampling scheme. Following the discussion in \cref{subsection:sampling-times}, the sampling scheme should depend on the hardware environment and the problem setting. Despite this, the two sampling patterns studied in \cref{subsection:sampling-times} form a foundation for most sampling schemes in practice and warrant a precise analysis. After this analysis, certain adaptive schemes have become popular and are also analyzed in a generic manner. We will focus on the case of row-action methods (corresponding to \cref{theorem:no-mem-convergence}) as the column-action results (corresponding to \cref{theorem:no-mem-column-convergence}) are nearly identical.

The first result provides a proof of convergence when we sample without replacement from a finite population. We note that the result is quite general and does not depend on the nature of the sampling without replacement or the dependency of the samples whenever the finite population is exhausted. As a result, the bounds are loose, which may be unsatisfying. Should particular sampling patterns become sufficiently important to warrant a more detailed analysis, we will do so in future work.

\begin{proposition} \label{theorem-no-mem-w-o-replacement}
Let $w$ and $\lbrace w_\ell : \ell +1 \in \mathbb{N} \rbrace$ be defined as in \cref{lemma: rand perm sampling}. Then, under the setting of \cref{theorem:no-mem-convergence}, 
\begin{enumerate}
\item $\tau_\ell - \tau_{\ell-1} \leq 2N$ for all $\ell \in \mathbb{N}$, and 
\item $\lim_{\ell \to \infty} \prod_{j=1}^\ell \gamma_j = 0$. 
\end{enumerate}
Moreover, $\gamma_j$ are uniformly bounded by $\gamma \in [0,1)$ that depends on $\lbrace A'W_1,\ldots,A'W_N \rbrace$. Therefore, with probability one,
\begin{equation}
\norm{ x_{2N\ell} - x^* - P_{\mathcal{N}(w)}(x_0 - x^*) }_2^2 \leq \gamma^\ell \norm{ P_{\mathcal{R}(w)}(x_0 - x^*) }_2^2.
\end{equation}
\end{proposition}
\begin{proof}
By the definition of $w$ in \cref{lemma: rand perm sampling}, $\mathcal{R}(w) = \linspan{ A'W_1,\ldots,A'W_N }$. Moreover, by the definitions of $\lbrace w_\ell \rbrace$, we are sampling from $W_1,\ldots,W_N$ without replacement. Then, we are guaranteed that $\lbrace A'w_{\tau_{\ell-1}+1},\ldots,A'w_{\tau_{\ell}} \rbrace$ spans $\mathcal{R}(w)$ if $\lbrace W_1,\ldots,W_N \rbrace \subset \lbrace w_{\tau_{\ell-1}+1},\ldots,w_{\tau_\ell} \rbrace$. Now, suppose that at iteration $\tau_{\ell-1}$, $\mathcal{W} \subset \lbrace W_1,\ldots,W_N \rbrace$ are exhausted. Then, to ensure that $\lbrace W_1,\ldots,W_N \rbrace$ is contained in $\lbrace w_{\tau_{\ell-1}+1},\ldots,w_{\tau_\ell} \rbrace$, we need to exhaust $\mathcal{W}^c$ and then the entire set $\lbrace W_1,\ldots,W_N \rbrace$. Since $|\mathcal{W}^c| \leq N$, we need at most $2N$ more iterations from $\tau_{\ell-1}$ to achieve $\tau_{\ell}$. Therefore, $\tau_{\ell} - \tau_{\ell-1} \leq 2N$. Now, let $\mathcal{F}$ denote all matrices whose columns are maximal linearly independent subsets of
\begin{equation}
\left\lbrace \frac{A'W_1}{\norm{A'W_1}_2},\ldots,\frac{A'W_N}{\norm{A'W_N}_2} \right\rbrace.
\end{equation}
Then, $\mathcal{F}_\ell \subset \mathcal{F}$. Therefore,
\begin{equation}
\gamma_\ell = 1 - \min_{F \in \mathcal{F}_\ell} \det(F'F) \leq 1 - \min_{F \in \mathcal{F}} \det(F'F) =: \gamma.
\end{equation}
It is clear, by Hadamard's inequality, that $\gamma \in [0,1)$. Hence, $\lim_{\ell \to \infty} \prod_{j=1}^\ell \gamma_j \leq \lim_{\ell \to \infty} \gamma^\ell = 0$. The result follows by \cref{theorem:no-mem-convergence}.
\end{proof}

It is worth pausing here to compare our approach in \cref{theorem-no-mem-w-o-replacement} to previous results for cyclic row-action methods (e.g., \cite{kaczmarz1993},\footnote{This is a translated copy of Kaczmarz's original article, which is published in German \citep{karczmarz1937}.} algebraic reconstruction technique \citep{gordon1970}, cyclic block Kaczmarz). Our use of Meany's inequality to analyze such methods is not novel: Meany's inequality has been used previously to analyze deterministic row-action methods \citep{galantai2005,bai2013,wallace2014} with even more sophisticated refinements of Meany's inequality than what we have here, and a detailed comparison of Meany's inequality and other approaches to analyzing these deterministic variants can be found in \cite{dai2015}. However, our use of Meany's inequality generalizes these deterministic approaches as it (1) allows for an arbitrary transformation (via $\lbrace W_1,\ldots,W_N \rbrace$) of the original system, which has borne out to be a fruitful approach vis-\`{a}-vis matrix sketching \cite{woodruff2014}; and (2) allows for the benefits of random cyclic sampling, which many have observed to be the most productive route in practice and there is mounting evidence in adjacent fields that random cyclic sampling does indeed have practical benefits \citep{lee2019,wright2020}.
While our generalizations are valuable, further improvements are to be found by marrying our randomization framework with the more nuanced refinements of Meany's inequality found in \cite{galantai2005} and \cite{bai2013}, which we leave to future efforts.

The next result revisits the case of independent and identically distributed sampling. The result makes intuitive sense as, for such a situation, we should expect the difference in the stopping times to be independent and identically distributed, which, results in the natural conclusion that $\gamma_\ell$ are also independent and identically distributed. Moreover, we show that eventually, the rate of convergence is almost controlled by $\E{\gamma_1}$ with probability one. We again stress here that the generality of the results naturally makes them quite loose, and we discuss this further after the result.

\begin{proposition} \label{theorem-no-mem-w-replacement}
Let $w$ and $\lbrace w_\ell : \ell+1 \in \mathbb{N} \rbrace$ be defined as in \cref{theorem: iid sampling}. Then, under the setting of \cref{theorem:no-mem-convergence}, $\tau_\ell < \infty$ almost surely for all $\ell \in \mathbb{N}$, and $\lbrace \gamma_\ell : \ell \in \mathbb{N} \rbrace$ are independent and identically distributed such that 
$\E{\gamma_1} = 1 - \E{ \min_{ F \in \mathcal{F}_1} \det(F'F)} < 1$.
Hence, for all $\ell \in \mathbb{N}$ and $\delta > 1$,
\begin{equation}
\Prb{\bigcup_{j=1}^\infty \bigcap_{\ell=j}^\infty \left\lbrace \norm{ x_{\tau_\ell+1} - x^* - P_{\mathcal{N}(w)}(x_0 - x^*) }_2^2 \leq \E{ \gamma_1}^{\frac{\ell}{\delta}}  \norm{ P_{\mathcal{R}(w)}(x_0 - x^*) }_2^2 \right\rbrace} =1,
\end{equation}
where $\E{ \gamma_\ell } \in [0,1)$. Moreover, $\lim_{ \ell \to \infty} \tau_\ell/\ell = \E{ \tau_1}$.
\end{proposition}
\begin{remark}
In the proof below, we also compute the probability for each $j$ for which the conclusion of the preceding result holds. Thus, we can also make the usual ``high-probability'' statements without any additional effort. 
\end{remark}
\begin{proof}
Again, our main workhorse will be \cite[Theorem 4.1.3]{durrett2010}. By this result, conditioned on $\tau_{\ell-1}$, $\lbrace A'w_{\tau_{\ell-1}+1},A'w_{\tau_{\ell-1}+2},\ldots \rbrace$ are independent and identically distributed. By this property, conditioned on $\tau_{\ell-1}$, $\tau_{\ell} - \tau_{\ell-1}$ is independent of $\tau_{\ell-1}$ and have the same distribution for all $\ell \in \mathbb{N}$. We conclude then that since $\gamma_\ell$ is a function of $\lbrace A'w_{\tau_{\ell-1}+1},\ldots,A'w_{\tau_\ell} \rbrace$, then $\gamma_\ell$ are independent and identically distributed. We now conclude that \cref{theorem-eqn:no-mem-stop-rate} holds with probability one by applying \cref{theorem:no-mem-convergence}. For any $\delta > 1$, by Markov's inequality and independence,
\begin{equation}
\Prb{ \prod_{j=1}^\ell \gamma_j > \E{ \gamma_1}^{k/\delta} } \leq \left(\E{ \gamma_1}^{1 - \frac{1}{\delta}} \right)^k.
\end{equation}
Since $\E{ \gamma_1}^{1 - \frac{1}{\delta}} < 1$, the Borel-Cantelli lemma implies that the probability that the product of $\gamma_j$ is eventually less than $\E{\gamma_1}^{k/\delta}$ is one.
\end{proof}

Here, we again take a moment to compare this result to the results of \cite{richtarik2017}. Namely, we are interested in how the rate of convergence of \cref{theorem-no-mem-w-replacement} compares with the rate of convergence result in \cite{richtarik2017}. To make this comparison, we numerically estimate the theoretical rates of convergence proposed by our result and the result of \cite{richtarik2017} on five matrices from the {\tt MatrixDepot} (as described in \cref{section: experiments}). We show these comparisons in \cref{table:iid-rate-comparison}.  We show these comparisons in \cref{table:iid-rate-comparison}. As expected, the results of \cite{richtarik2017}, which are specialized to the i.i.d. case and apply on average, are much tighter than our general results that apply to more than just i.i.d. case and hold with probability one.

\begin{table}[!htb]
\begin{footnotesize}
\caption{A comparison in the estimated theoretical bounds on the rates of convergence of Gaussian-sketched base randomized methods in $\ell^2$ between this work and the results in \cite{richtarik2017}. The estimates are made by simulation of the theoretical rates. The comparison is made on five different matrices available in the {\tt MatrixDepot}, as described in \cref{section: experiments}. The main message is that the results of \cite{richtarik2017} are tighter than our result, as they apply to the average case. This is expected as our result applies to more than just the i.i.d. sampling case and hold with probability one (asymptotically).}
\label{table:iid-rate-comparison}

\begin{center}
\renewcommand{\arraystretch}{1.3}
\begin{tabular}{@{}lrcrc@{}} \toprule
\multicolumn{5}{c}{\textbf{Comparison of Estimated Theoretical Rates of Convergence}} \\ \midrule 
\multicolumn{1}{c}{Matrix Name} & & \multicolumn{3}{c}{Estimated Rates by Result} \\ 
\cmidrule{3-5}
  & & Theorem 4.8 of \cite{richtarik2017} &  & \cref{theorem-no-mem-w-replacement} \\ \midrule 
deriv2 & & $1 - \bigO{10^{-4}}$ & & $1 - \bigO{10^{-35}}$ \\ 
heat   & &  $1-\bigO{10^{-15}}$ & & $1 - \bigO{10^{-34}}$\\
randsvd & &  $1-\bigO{10^{-15}}$& & $1 - \bigO{10^{-71}}$\\
ursell & & $1-\bigO{10^{-16}}$ & & $1 - \bigO{10^{-161}}$\\
wing & & $1-\bigO{10^{-16}}$ & & $1 - \bigO{10^{-163}}$\\
\bottomrule
\end{tabular}
\end{center}
\end{footnotesize}
\end{table}

\subsection{Adaptive Sampling Schemes} \label{subsection:adaptive-sampling}

To bookend this section, we discuss how our results can be applied to a broad set of adaptive methods that make use of the residual information at a given iterate whether deterministically (e.g., \cite{motzkin1954,gubin1967,lent1976,censor1981}) or randomly (e.g., \cite{nutini2016,bai2018,haddock2019}). In \cref{subsubsection:framework-adaptive}, we will begin with some formalism to establish a general class of adaptive methods, and we then prove convergence and a rate of convergence for such methods. In \cref{subsubsection:specific-examples-adaptive}, we provide concrete examples at the end.

\subsubsection{A General Class and Analysis of Adaptive Methods} \label{subsubsection:framework-adaptive}
To be rigorous, let $x_0 \in \mathbb{R}^d$ and let $\varphi: (A,b,\lbrace x_j : j \leq k \rbrace) \mapsto w_k$ be an adaptive procedure for generating $\lbrace w_k \rbrace$ according to the following procedure: for $k +1 \in \mathbb{N}$, 
\begin{equation} \label{eqn:adaptive-procedure}
\begin{aligned}
w_k &= \varphi(A, b , \lbrace x_j : j \leq k \rbrace) \\
x_{k+1} & = x_{k} + \frac{A'w_k w_k'(b - Ax_k)}{\norm{A' w_k}_2^2}.
\end{aligned}
\end{equation}
\begin{remark}
While we will focus on the base methods of type \cref{eqn:no-memory-iteration}, methods of the type \cref{eqn:no-memory-column-iteration} can be handled analogously.
\end{remark}

While \cref{eqn:adaptive-procedure} is quite general, the vast majority of adaptive schemes make further restrictions that we abstract in the following definitions. 

\begin{definition}[Markovian] \label{defn:markovian}
For a fixed integer $\eta$, an adaptive procedure, $\varphi$, is $\eta$-Markovian if the conditional distribution of $\varphi(A,b, \lbrace x_j : j \leq k \rbrace) $ given $\lbrace x_j : j \leq k \rbrace$ is equal to the conditional distribution of $\varphi(A, b, \lbrace x_j : j \leq k \rbrace)$ given $\lbrace x_j : k - \eta < j \leq k \rbrace$. If a procedure is $1$-Markovian, we will frequently call it Markovian.
\end{definition}

A consequence of the $\eta$-Markovian property is that we can write $\varphi (A , b , \lbrace x_j : j \leq k \rbrace) $ as $\varphi (A, b, \lbrace x_j : k - \eta < j \leq k \rbrace )$. In the case of a $1$-Markovian adaptive procedure, we will simply write $\varphi(A, b, x_k )$. The $1$-Markovian property is readily satisfied for a number of common procedures analyzed in the literature (e.g., maximum residual, maximum distance, etc.), which may suggest that the $\eta$-Markovian notion is irrelevant for general $\eta$. We contend though, that procedures that are memory-sensitive may be more apt to make use of the $\eta$-Markovian property for $\eta > 1$. For example, to demonstrate its potential value, consider a procedure that selects the equations with the top $\eta$ residuals, pulls them into memory, and simply cycles through them deterministically or randomly. Then this simple procedure would be $\eta$-Markovian. However, owing to the lack of such procedures in the literature, we will focus on the $1$-Markovian case for which we can write $\varphi(A, b , x)$, and note that the results and definitions are readily extendable.

The next definition establishes another key property of these adaptive schemes that rely on residuals. 

\begin{definition}[Magnitude Invariance] \label{defn:mag-invar}
Let $H$ represent the set of solutions to $Ax = b$, and let $P_H : \mathbb{R}^d \to H$ represent the projection of a vector onto $H$,\footnote{Since $H$ is a flat, $P_H$ is not guaranteed to be a linear operator.} then an adaptive procedure, $\varphi$, is magnitude invariant if, for any $x \not\in H$ and any $\lambda > 0$, the distribution of $\varphi( A, b, x )$ is equal to the distribution of 
\begin{equation}
\varphi ( A, b, P_H(x) + \lambda [ x - P_H(x) ] ).
\end{equation}
\end{definition}

The magnitude invariance of a number of adaptive methods often follows from the following simple calculation that we state as a lemma for future reference.

\begin{lemma} \label{lemma:residual-mag-invar}
Let $x \in \mathbb{R}^d$ and let $v_1, v_2 \in \mathbb{R}^n$. Then, for any $\lambda > 0$, if $ | v_1' (A x - b) | \geq |v_2' (Ax - b)|$ then
\begin{equation}
| v_1' (A ( P_H(x) + \lambda [ x - P_H(x) ] ) - b) | \geq | v_2' (A ( P_H(x) + \lambda [ x - P_H(x) ] ) - b) |.
\end{equation}
If the hypothesis holds with a strict inequality, then so does the conclusion.
\end{lemma}
\begin{proof}
Note, $A P_H(x) = b$. Therefore, $ A (P_H(x) + \lambda [ x - P_H(x) ] ) - b = \lambda (A x - b)$. From the hypothesis and $\lambda > 0$, $\lambda | v_1' (Ax - b) | \geq \lambda | v_2'(Ax - b)|$. Also owing to $\lambda > 0$, we can replace the inequalities with strict inequalities.
\end{proof}

Furthermore, the magnitude invariance property has hidden within it an additional feature: the projection of $x$ onto the null space is irrelevant (as we might expect for a procedure depending on the residual). As a result, we can, without losing generality, focus our discussion to $x$ that are in the row space of $A$, which has a unique intersection with $H$ at a point that we denote $x_{\row}^*$. Furthermore, the magnitude invariance property allows us to focus specifically on the Euclidean unit sphere around $x_{\row}^*$, which we denote by $\mathbb{S}(x_{\row}^*)$. This will be essential to the next definition.

The final definition ensures that if \cref{eqn:adaptive-procedure} makes too much progress along one particular subspace, then it must have a nonzero probability of exploring an orthogonal subspace relative to, roughly, the row space of $A$. Before stating this definition, we need to be slightly careful here with using the row space of $A$: if the rows of $A$ can be partitioned into two sets that are mutually orthogonal and $x_0$ is initialized in the span of one of these subsets, then we will never need to visit the other set and, consequently, we will never observe the entire row space of $A$. To account for this, we can focus on the restricted row space,
\begin{equation} \label{eqn:restricted-row}
\rrow(A) = \mathrm{span}[ A_{i,\cdot} : A_{i,\cdot}'x_0 \neq b_i ].
\end{equation}
This definition may seem unnecessary as we can account for this (more generally) via $\mathcal{R}(w)$ by an appropriate choice of $w$. However, in our previous statements, we defined $w$ before specifying $x_0$. Here, we would need to know $x_0$ in order to define $w$ and, thus, $\mathcal{R}(w)$ appropriately. Fortunately, an examination of the preceding results shows that this ordering is not important and the results hold even if $w$ is defined given $x_0$ or even future iterates. With this explanation in hand, we can now state the final definition.

\begin{definition}[Exploratory] \label{defn:exploratory}
Let $x_0 \in \mathbb{R}^d$ and define $\rrow(A)$ accordingly. An adaptive procedure, $\varphi$, is exploratory if for any proper subspace $V \subsetneq \rrow(A)$, there exists $\pi \in (0,1]$ such that
\begin{equation} \label{eqn-defn-exploratory}
\sup_{ x \in \mathbb{S}(x_{\row}^*) \cap V } \Prb{ A' \varphi(A, b, x) \perp V } \leq 1 - \pi.
\end{equation}
\end{definition}

\begin{remark}
If magnitude invariance does not hold, then we could specify the exploratory property to hold for any point in $V$ that is distinct from $x_{\row}^*$. For this modified definition of the exploratory property, the results below would still hold. Then, why should we keep the magnitude invariance property? It is out of practicality. The magnitude invariance property allows us to restrict the verification of the exploratory property to the unit ball, and then we can apply it to any iterate regardless of its distance to the solution.
\end{remark}

For a Markovian, magnitude invariant and exploratory adaptive scheme, $\varphi$, we will need one assumption before stating the result.
\begin{assumption} \label{assumption:max-convergence}
Let $\mathcal{F}$ denote the set of matrices whose columns are normalized, maximal linearly independent subsets of 
\begin{equation}
\left\lbrace A' \varphi(A, b, x_1),\ldots, A' \varphi(A, b, x_d) \right\rbrace,
\end{equation}
where $x_1,\ldots,x_d \in \mathbb{R}^d$ are arbitrary vectors. Suppose, for this choice of $\varphi$,
\begin{equation}
1 - \inf_{ F \in \mathcal{F} } \det( F' F) =: \gamma \in [0,1).
\end{equation}
\end{assumption}

\begin{remark}
As we will see, \cref{assumption:max-convergence} is sufficient for us to uniformly treat the many examples in the literature that are selecting equations or, more generally, are of the form in \cref{lemma: rand perm sampling}, rather than generating linear combinations of them. In the case of linear combinations, we could refine this assumption to account for the nature of the linear combinations as we do in \cref{theorem-no-mem-w-replacement}.
\end{remark}

\begin{theorem} \label{theorem:adaptive-convergence}
Suppose $Ax = b$ admits a solution $x^*$ (not necessarily unique); let $H$ denote the set of all solution, and $P_H$ be the projection onto this flat. Let $x_0 \in \mathbb{R}^d$ and let $\rrow(A)$ be defined as above (see \cref{eqn:restricted-row}). Moreover, let $\varphi$ be a $1$-Markovian, magnitude invariant and exploratory adaptive procedure satisfying \cref{assumption:max-convergence} that generates $\lbrace x_k \rbrace$ and $\lbrace w_k \rbrace$ according to \cref{eqn:adaptive-procedure} and so that $\Prb{ A'w_k \in \rrow(A) } = 1$ for all $k + 1 \in \mathbb{N}$. Then, there exist an increasing sequence of stopping times $\lbrace \tau_\ell : \ell \in \mathbb{N} \rbrace$ such that $\Prb{ E_1 \cup E_2 } = 1$, where:
\begin{enumerate}
\item $E_1$ is the event of iterates that terminate finitely to a solution of $Ax=b$; that is,
\begin{equation}
E_1 = \bigcup_{\ell \in \mathbb{N}} \left\lbrace x_{\tau_{\ell} + 1} \in H \right\rbrace.
\end{equation}
\item $E_2$ is the event of iterates that infinitely converge to a solution of $Ax=b$; that is,
\begin{equation}
E_2 = \bigcap_{\ell \in \mathbb{N}} \left\lbrace \norm{ x_{\tau_{\ell} + 1} - P_H(x_0) }_2^2 \leq \gamma^\ell \norm{ x_0 - P_H(x_0) }_2^2 \right\rbrace.
\end{equation}
Moreover, on $E_1$, $\tau_\ell$ has finite expectation for $\ell$ such that $x_{\tau_{\ell}+1} \in H$. Similarly, on $E_2$, $\tau_{\ell}$ has finite expectation for all $\ell$.
\end{enumerate}

\end{theorem}
\begin{proof}
Without loss of generality, we will assume $x_0 \in \row(A)$. We will consider the nontrivial case where $x_0 \neq x_{\row}^*$. Note, by the construction of $\rrow(A)$, it must hold then $x_0 - x_{\row}^* \in \rrow(A)$.  To prove the result, we will make three claims of the following rough nature and purpose, which we will make precise below.
\begin{enumerate}
\item Finite termination can only occur at a point $x_{k+1}$ if and only if $A'\varphi(A,b,x_k)$ is parallel to $x_k - x_{\row}^*$. We will use this claim to specify the set $E_1$.
\item For the first time the span of the iterate errors,  $\mathrm{span}[\lbrace x_k - x_{\row}^* \rbrace]$, fails to (non-trivially) increase in dimension, the corresponding $\lbrace A'w_k \rbrace$ up to this iterate span the subspace. As a result, with an appropriate definition of $\mathcal{R}(w)$, we will apply \cref{theorem:no-mem-convergence} to prove a multiplicative decrease in the iterate errors by a factor of $\gamma$.
\item Finally, we show that the first time that the span of the iterate errors fails to (non-trivially) increase in dimension must be finite with probability one and have bounded expectation. By combining the first claim with this claim, we have the property specified by the event $E_1$. By combining this claim with the second claim, we have the property specified by the event $E_2$. By this claim alone, we have that $\Prb{E_1 \cup E_2} = 1$.
\end{enumerate}

To establish our claims, we need some additional notation. Let $\xi$ be an arbitrary finite stopping time and define
\begin{equation}
 V_k = \mathrm{span}\left[ x_{\xi} - x_{\row}^*,x_{\xi+1} - x_{\row}^*,\ldots, x_{\xi+k} - x_{\row}^* \right],
\end{equation}
and $V_{k}^0 = \mathrm{span} \left[ x_{\xi+k} - x_{\row}^* \right]$.
Furthermore,  define 
\begin{equation}
\nu = \min \left\lbrace k \geq 0: x_{\xi + k + 1} - x_{\row}^* \in V_{k}, x_{\xi+k+1} \neq x_{\xi+k} \right\rbrace.
\end{equation} 
Note, $\nu$ corresponds to the first time that the span of the iterate errors, starting at $\xi$, fails to non-trivially increase in dimension. It will often be more succinct to specify the non-trivial cases by an indicator variable given by
\begin{equation}
\chi_{\xi+k} = \1{ \varphi(A,b,x_{\xi+k})'A(x_{\xi+k} - x_{\row}^*) \neq 0}.
\end{equation}
By \cref{eqn:adaptive-procedure}, we can readily replace $x_{\xi+k+1} \neq x_{\xi+k}$ in the definition of $\nu$ with $\chi_{\xi+k} = 1$. We now state and prove our claims precisely.

\underline{Claim 1:}
Suppose $x_{\xi} - x_{\row}^* \neq 0$. We claim that $x_{\xi+1} = x_{\row}^*$ if and only if $A'\varphi(A,b,x_\xi) \in V_0 \setminus \lbrace 0 \rbrace$. 

Note, this claim readily follows from
\begin{equation}
x_{\xi+1} - x_{\row}^* = x_{\xi} - x_{\row}^* - \frac{A' \varphi(A,b,x_\xi) \varphi(A,b,x_\xi)'A}{\norm{ A'\varphi(A,b,x_\xi) }_2^2} (x_\xi - x_{\row}^*),
\end{equation}
which, in turn, follows from \cref{eqn:adaptive-procedure}.

\underline{Claim 2:} Suppose $\nu$ is finite and define $V_{\nu}$. We claim that
\begin{equation}
\mathrm{span}\left[ A'\varphi(A, b, x_{\xi} ) \chi_{\xi},\ldots, A'\varphi(A, b, x_{\xi+\nu} ) \chi_{\xi +\nu} \right] = V_{\nu}.
\end{equation}

We first note that $A'\varphi(A, b, x_{\xi + k})\chi_{\xi+k} \in V_{\nu}$ for any $k \in [0,\nu]$ by \cref{eqn:adaptive-procedure}. Therefore, we see that the span of $\Phi = \lbrace A'\varphi(A, b, x_{\xi} )\chi_{\xi},\ldots, A'\varphi(A,b,x_{\xi+\nu}) \chi_{\xi+\nu} \rbrace$ is contained in $V_{\nu}$. To show that $V_{\nu}$ is included in the span of $\Phi$, note that, by the definition of $V_{\nu}$ and by \cref{eqn:adaptive-procedure},
\begin{equation} \label{eqn-proof:repeating-subspace}
V_{\nu} = \mathrm{span}\left[ A'\varphi(A, b, x_{\xi}) \chi_{\xi},\ldots, A'\varphi(A,b,x_{\xi+\nu-1})\chi_{\xi+\nu-1}, x_{\xi+\nu} - x_{\row}^* \right].
\end{equation}
Moreover, the nonzero terms on the generating set on the right hand side of \cref{eqn-proof:repeating-subspace} must be linearly independent, as anything else would contradict the minimality of $\nu$. We are left to show that $x_{\xi+\nu} - x_{\row}^*$ is in the span of $\Phi$. To do this, we perform Gram-Schmidt on the generating set in \cref{eqn-proof:repeating-subspace} starting with $x_{\xi+\nu} - x_{\row}^*$. Denote the remaining vectors in this set $\phi_1,\ldots,\phi_{r-1}$ where $r = \dim( V_{\nu} )$. Then, by the definition of $\nu$, $x_{\xi+\nu+1} - x_{\row}^* \in V_{\nu}$. Therefore, there exist constants $c_0,\ldots,c_{r-1}$ such that
\begin{equation}
\begin{aligned}
& c_0 (x_{\xi+\nu} - x_{\row}^*) + \sum_{j=1}^{r-1} c_j \phi_j \\
&\quad = x_{\xi+\nu} - x_{\row}^* - \frac{A' \varphi(A, b, x_{\xi+\nu})\varphi(A, b, x_{\xi+\nu})'A }{\norm{A'\varphi(A, b, x_{\xi+\nu})}_2^2} ( x_{\xi + \nu} - x_{\row}^* ). 
\end{aligned}
\end{equation} 
If $c_0 \neq 1$, we see that the claim follows. For a contradiction, suppose that $c_0 = 1$. Then $A'\varphi(A,b,x_{\xi+\nu})$ can be written as a linear combination of vectors that are orthogonal to $x_{\xi+\nu} - x_{\row}^*$. This would imply then that $\chi_{\xi+\nu} = 0$, which contradicts the definition of $\nu$. Hence, we see that the claim holds.

\underline{Claim 3:} For any finite stopping time $\xi$, $\nu$ is finite with probability one and has bounded expectation. 

To show this, we define a sequence of stopping times. Define 
\begin{equation}
s_1 = \min \left\lbrace k : \chi_{\xi+k} \neq 0 \right\rbrace,
\end{equation}
and
\begin{equation}
s_{j} = \min \left\lbrace k : \chi_{\xi+s_1+\cdots+s_{j-1} + k } \neq 0 \right\rbrace.
\end{equation}
By the definition of $\nu$, $\nu$ can only take values in $\lbrace \sum_{i=1}^j s_i : j \in \mathbb{N} \rbrace$. Moreover, at each $s_j$, we must either observe $\lbrace \dim( V_{\xi+ s_1 + \cdots + s_j + 1} ) = \dim( V_{\xi + s_1 + \cdots + s_j}) + 1  \rbrace$ or $\lbrace \nu \leq \sum_{i=1}^j s_i \rbrace$. Hence, at most, we see that $\nu$ can only take values in $\lbrace \sum_{i=1}^j s_i : j = 1,\ldots, r \rbrace$ where $r = \dim( \rrow(A) )$. Thus, if we show that each $s_j$ is finite and has bounded expectation, then $\nu$ must be finite and have bounded expectation. By the magnitude invariance, Markovian and exploratory properties, we conclude that
\begin{equation} \label{eqn:s_j}
\begin{aligned}
&\condPrb{s_j = k}{\xi, s_1,\ldots,s_{j-1}, x_{\xi},\ldots,x_{\xi+s_1+\cdots+s_{j-1}+1 } } \\
&\quad \leq ( 1- \pi( V_{s_1+\cdots+s_{j-1} + 1}) )^{k-1} \pi( V_{s_1+\cdots+s_{j-1} + 1}).
\end{aligned}
\end{equation}
Therefore, we see that $s_j$ is finite and has bounded expectation.

\underline{Conclusion:} From these three claims we can now prove the result by induction. 
\paragraph{Base Case} Define $\mathfrak{E}_0^c = \lbrace x_0 \neq x_{\row}^* \rbrace$. On this event, if we take $\xi = 0$ and define $\tau_1$ to be the corresponding $\nu$. On $\mathfrak{E}_0^c$, $\tau_1$ is finite and has finite expectation by Claim 3. Then, we can define, as a subset of $\mathfrak{E}_0^c$,
\begin{equation}
\mathfrak{E}_1 = \lbrace A'\varphi(A,b,x_{\tau_1}) \in V_{\tau_1}^0  \setminus \lbrace 0 \rbrace \rbrace,
\end{equation}
and $\mathfrak{E}_1^c$ to be its relative complement on $\mathfrak{E}_0$.

Note, 
\begin{enumerate}
\item By Claim 1, $\mathfrak{E}_1$ is equivalent to the event $x_{\tau_1 + 1} = x_{\row}^*$ up to a measure zero set. 
\item By Claim 2, \cref{theorem:no-mem-convergence} with $\mathcal{R}(w) = V_{\tau_1}$, and \cref{assumption:max-convergence}, $\mathfrak{E}_1^c$ is contained in the event on which
\begin{equation}
\norm{ x_{\tau_1 + 1} - x_{\row}^* }_2^2 \leq \gamma \norm{ x_0 - x_{\row}^*}_2^2
\end{equation}
up to a measure zero set. 
\end{enumerate}

\paragraph{Induction Hypothesis} Let $\ell \in \mathbb{N}$. On the event $\mathfrak{E}_{\ell-1}^c$, we let $\xi = \tau_{\ell-1} + 1$ and, for the correspondingly defined $\nu$, we can define $\tau_{\ell} = \tau_{\ell-1} + 1 + \nu$. Furthermore, on $\mathfrak{E}_{\ell-1}^c$, $\tau_{\ell}$ is finite and has finite expectation. We can define, as a subset of $\mathfrak{E}_{\ell-1}^c$,
\begin{equation}
\mathfrak{E}_{\ell} = \lbrace A' \varphi(A, b, x_{\tau_{\ell}}) \in V_{\tau_{\ell}}^0 \setminus \lbrace 0 \rbrace\rbrace,
\end{equation}
and $\mathfrak{E}_{\ell}^c$ to be its relative complement on $\mathfrak{E}_{\ell-1}^c$. 

Further,
\begin{enumerate}
\item $\mathfrak{E}_{\ell}$ is equivalent to the event $x_{\tau_{\ell}+1} = x_{\row}^*$ up to a measure zero set.
\item $\mathfrak{E}_{\ell}^c$ is contained in the event on which
\begin{equation}
\norm{ x_{\tau_{\ell} + 1} - x_{\row}^* }_2^2 \leq \gamma \norm{ x_{\ell} - x_{\row}^*}_2^2
\end{equation}
up to a measure zero set.
\end{enumerate}

\paragraph{Generalization} On the event $\mathfrak{E}_{\ell}^c$, we let $\xi = \tau_{\ell}+1$ and, for the correspondingly defined $\nu$, we can define $\tau_{\ell+1} = \tau_{\ell} + 1 + \nu$. On $\mathfrak{E}_{\ell}^c$, $\tau_{\ell+1}$ is finite and has finite expectation by Claim 3. Then, we can define, as a subset of $\mathfrak{E}_{\ell}^c$,
\begin{equation}
\mathfrak{E}_{\ell+1} = \lbrace A' \varphi(A, b, x_{\tau_{\ell+1}}) \in V_{\tau_{\ell+1}}^0 \setminus \lbrace 0 \rbrace\rbrace,
\end{equation}
and  $\mathfrak{E}_{\ell+1}^c$ to be its relative complement on $\mathfrak{E}_{\ell}^c$.

\begin{enumerate}
\item By Claim 1, $\mathfrak{E}_{\ell+1}$ is equivalent to the event $x_{\tau_{\ell+1} + 1} = x_{\row}^*$ up to a measure zero set. 
\item By Claim 2, \cref{theorem:no-mem-convergence} with $\mathcal{R}(w) = V_{\tau_{\ell+1}}$, and \cref{assumption:max-convergence}, $\mathfrak{E}_{\ell+1}^c$ is contained in the event on which
\begin{equation}
\norm{ x_{\tau_{\ell+1} + 1} - x_{\row}^* }_2^2 \leq \gamma \norm{ x_{\tau_\ell} - x_{\row}^*}_2^2
\end{equation}
up to a measure zero set. 
\end{enumerate}

Therefore, by the induction claims,
\begin{equation}
E_1 = \bigcup_{\ell \in \mathbb{N}} \mathfrak{E}_{\ell}
\end{equation}
and
\begin{equation}
E_2 = \bigcap_{\ell \in \mathbb{N}} \mathfrak{E}_{\ell}^c,
\end{equation}
and $\Prb{ E_1 \cup E_2} =1$.
\end{proof}

\subsubsection{Applying our General Theory to Specific Adaptive Schemes} \label{subsubsection:specific-examples-adaptive}

To demonstrate the utility of \cref{theorem:adaptive-convergence}, we show that a number of classical and recent methods satisfy \cref{defn:mag-invar,defn:markovian,defn:exploratory,assumption:max-convergence}. In fact, we will show that a stronger version of \cref{defn:exploratory} holds for these methods, which allows us to explicitly upper bound the elements of $\lbrace \E{\tau_{\ell}}: \ell \in \mathbb{N} \rbrace$ (when they are defined). 

\begin{proposition} \label{proposition:convergence-specific-adaptive}
Suppose $Ax = b$ admits a solution $x^*$. Let $x_0 \in \mathbb{R}^d$ and let $\rrow(A)$ be defined as above (see \cref{eqn:restricted-row}). Suppose that we define $\lbrace x_k \rbrace$ and $\lbrace w_k \rbrace$ according to \cref{eqn:adaptive-procedure} for the following adaptive methods
\begin{enumerate}
\item the maximum residual method \citep[see][Section 4]{agmon1954});
\item the maximum distance method \citep[see][Section 3]{agmon1954};
\item the Greedy Randomized Kaczmarz method \citep[see][Method 2]{bai2018};
\item the Sampling Kaczmarz-Motzkin method \citep[see][Page 4]{haddock2019}.
\end{enumerate}
Then, for each of the above methods, there exists a $\gamma \in [0,1)$ such that the conclusions of \cref{theorem:adaptive-convergence} hold. Moreover, there exists a constant $\kappa$ such that for any finite $\tau_{\ell}$ (as specified in \cref{theorem:adaptive-convergence}), $\E{\tau_{\ell}} \leq \ell \kappa$.
\end{proposition}
\begin{remark}
Greedy Randomized Kaczmarz is an example of methods that deterministically determine a threshold over residuals; select the equations whose residuals surpass this threshold; and then randomly select from this set. For this more general class, so long as the threshold satisfies the magnitude invariance property and the random selection does not give any equation less than zero probability, then the result applies to this more general class. Similarly, Sampling Kaczmarz-Motzkin is an example of methods that randomly determine a set of equations; and then deterministically select from this subset of equations based on the residual values. So long as the random subset of equations does not give any equation less than zero probability (that is not already satisfied), then the result will apply to this more general class as well.
\end{remark}
\begin{remark}
Our partial orthogonalization methods (see \cref{alg: rank-one RPM low mem}) do not satisfy the $\eta$-Markovian property, as the partial orthogonalizations have a dependence on every preceding iterate.
\end{remark}

For each method, we show that it satisfies \cref{defn:mag-invar,defn:markovian,defn:exploratory,assumption:max-convergence}. In fact, for each method, we will show that a stronger version of \cref{defn:exploratory} holds.
We will start by establishing several general facts that will be useful in the discussion of each method.

\begin{lemma} \label{lemma:min-max-inner-product}
Let $x_0 \in \row(A)$ and define $\rrow(A)$ as in \cref{eqn:restricted-row}. Then,
\begin{equation}
\inf_{ v \in \rrow(A) \cap \mathbb{S}(0) }  \max _{ i \in \lbrace 1, \ldots, n \rbrace} \frac{|A_{i,\cdot}'v|}{\norm{A_{i,\cdot}}_2} =: c > 0,
\end{equation}
where $\mathbb{S}(0)$ is the Euclidean unit sphere around the zero vector.
\end{lemma}

\begin{proof}
For each $v \in \rrow(A) \cap \mathbb{S}(0)$, we see that 
\begin{equation}
\max_{i \in \lbrace 1,\ldots, n \rbrace } \frac{|A_{i,\cdot}' v | }{\norm{ A_{i,\cdot} }_2 } =: c_v > 0,
\end{equation}
else $v \perp \rrow(A)$ and $v \in \rrow(A) \cap \mathbb{S}(0) \subset \rrow(A)$ and we would have a contradiction since $v \neq 0$. By continuity, we see that we can construct an open ball around each $v \in \rrow(A) \cap \mathbb{S}(0)$, $D_v$, such that 
\begin{equation}
\max_{i \in \lbrace 1, \ldots, n \rbrace } \frac{|A_{i,\cdot}'\tilde v|}{\norm{ A_{i,\cdot}}_2 } > c_v / 2,
\end{equation}
for all $\tilde v \in D_v \cap \mathbb{S}(0)$. Now, $\lbrace D_v : v \in \rrow(A) \cap \mathbb{S}(0) \rbrace$ is an open cover of $\rrow(A) \cap \mathbb{S}(0)$, which is a compact space. Hence, there is a finite subcover given by $\lbrace D_{v_1},\ldots,D_{v_K} \rbrace$. It follows that since each $v \in \rrow(A) \cap \mathbb{S}(0)$ belongs to one of the elements in the subcover, then
\begin{equation}
\inf_{ v \in \rrow(A) \cap \mathbb{S}(0) }  \max _{ i \in \lbrace 1, \ldots, n \rbrace} \frac{|A_{i,\cdot}'v|}{\norm{A_{i,\cdot}}_2} \geq \min \lbrace c_{v_1}/2,\ldots, c_{v_K}/2 \rbrace > 0.
\end{equation}
Therefore $c > 0$.
\end{proof}

\begin{lemma} \label{lemma:convergence-rate-bound}
Let $x_0 \in \row(A)$ and define $\rrow(A)$ as in \cref{eqn:restricted-row}. Let $\Phi = \left\lbrace A_{i,\cdot} : A_{i,\cdot} \in\rrow(A) \right\rbrace$. Let $\mathcal{F}$ be the matrices whose columns are normalized, maximal linearly independent vectors from $\Phi$. Then
\begin{equation}
1 - \min_{F \in \mathcal{F}} \det(F'F) =: \gamma < 1.
\end{equation}
\end{lemma}
\begin{proof}
There are only a finite number of matrices in $\mathcal{F}$ up to column permutations. Therefore, we can choose the $F \in \mathcal{F}$ that minimizes $\det(F'F)$. By Hadarmard's inequality, $\det(F'F) \in (0,1]$, which implies that $\gamma \in [0,1)$.
\end{proof}

\paragraph{Maximum Residual Method.} In the maximum residual method, $\varphi(A,b,x)$ is the standard basis vector in $\mathbb{R}^n$, $\lbrace e_1,\ldots,e_n\rbrace$, that solves
\begin{equation}
\max_{e \in \lbrace e_1,\ldots,e_n \rbrace} |e'(Ax - b)|.
\end{equation}

\underline{$1$-Markovian:} It follows from the definition of the maximum residual method that it only relies on the current iterate to evaluate $\varphi$. Therefore, it is $1$-Markovian.

\underline{Magnitude Invariance:} By \cref{lemma:residual-mag-invar}, it follows that $\varphi(A,b,x)$ is magnitude invariant.

\underline{Exploratory:} Consider any $A_{i,\cdot} \perp V$. Then, $0 = A_{i,\cdot}'(x - x_{\row}^*) = A_{i,\cdot}'x - b_i$. Therefore, we have that the only equations whose residuals are non-zero are the ones such that $P_{V} A_{i,\cdot} \neq 0$, and there is at least one such equation by \cref{lemma:min-max-inner-product}. Therefore,
\begin{equation}
\sup_{ x \in \mathbb{S}(x_{\row}^*) \cap V } \Prb{ A'\varphi(A,b,x) \perp V } = 0.
\end{equation}
That is, we satisfy the exploratory property in a stronger manner:
\begin{equation}
\sup_{V \subsetneq \rrow(A)} \sup_{ x \in \mathbb{S}(x_{\row}^*) \cap V } \Prb{ A'\varphi(A,b,x) \perp V } = 0.
\end{equation}

With these three properties verified and by \cref{lemma:convergence-rate-bound}, the conditions of \cref{theorem:adaptive-convergence} are satisfied and the result holds. The only thing left to show is that $\E{ \tau_{\ell}}$ are bounded by some $\ell \kappa$. By the proof of \cref{theorem:adaptive-convergence}, it is enough to bound the conditional expectations of $s_j$ in \cref{eqn:s_j}. Given that $\pi = 1$ for all $V \subsetneq \rrow(A)$, 
\begin{equation}
\condPrb{s_j = 1}{\xi, s_1,\ldots,s_{j-1}, x_{\xi},\ldots,x_{\xi+s_1+\cdots+s_{j-1}+1 } } = 1.
\end{equation}
Hence, $\nu \leq \dim(\rrow(A))$. Thus, $\E{\tau_{\ell}} \leq \ell \dim( \rrow(A))$. $\quad\blacksquare$

\paragraph{Maximum Distance Method.} In the maximum distance method, $\varphi(A,b,x)$ is the standard basis vector in $\mathbb{R}^n$ that solves
\begin{equation}
\max_{ e \in \lbrace e_1,\ldots,e_n \rbrace} \frac{|e' (Ax - b) |}{\norm{ A'e }_2^2}.
\end{equation}

\underline{$1$-Markovian:} It follows from the definition of the maximum distance method that it only relies on the current iterate to evaluate $\varphi$. Therefore, it is $1$-Markovian.

\underline{Magnitude Invariance:} Note, \cref{lemma:residual-mag-invar} still holds if we were to divide by the norm squared of $A_{i,\cdot}$. It follows that the maximum distance method is magnitude invariant.

\underline{Exploratory:} Just as in the maximum residual method, if $A_{i,\cdot}$ that is orthogonal to a subspace $V$, then $A_{i,\cdot}'x - b_i = 0$ for any $x \in V \cap \mathbb{S}(x_{\row}^*)$. Moreover, by \cref{lemma:min-max-inner-product}, there is at least one equation such that $A_{j,\cdot}'x - b \neq 0$ for all $x \in V \cap \mathbb{S}(x_{\row}^*)$. Hence, the maximum distance method satisfies a stronger version of the exploratory condition, namely,
\begin{equation}
\sup_{V \subsetneq \rrow(A)} \sup_{ x \in \mathbb{S}(x_{\row}^*) \cap V } \Prb{ A'\varphi(A,b,x) \perp V } = 0.
\end{equation}

By the same argument as above, \cref{theorem:adaptive-convergence} follows. Similarly, $\E{ \tau_{\ell}} \leq \ell \dim( \rrow(A) )$. $\quad\blacksquare$

\paragraph{Greedy Randomized Kaczmarz.} In \cite{bai2018} (Method 2), a residual threshold is selected given by 
\begin{equation} \label{eqn:grk-threshold}
\frac{1}{2}\left( \frac{1}{\norm{Ax - b}_2^2} \max_{ e \in \lbrace e_1,\ldots,e_n \rbrace} \frac{|e' (Ax - b) |^2}{\norm{ A'e }_2^2} + \frac{1}{\norm{A}_F^2} \right)
\end{equation}

Then, from the set of equations whose residual surpasses this threshold (which is shown to at least contain the equation selected by the maximum distance method), an equation is selected by a probability proportional to the equation's residual squared. 

\underline{$1$-Markovian:} Given that the threshold relies only on the current iterate value and that the random selection criteria only relies on the current residual, it follows that the Greedy Randomized Kaczmarz method is $1$-Markovian.

\underline{Magnitude Invariance:} Suppose $x \not\in H$. For $\lambda > 0$, let $x(\lambda) = P_H(x) + \lambda( x - P_{H}(x) )$. Then, by \cref{lemma:residual-mag-invar}, 
\begin{equation}
\begin{aligned}
&\frac{1}{\norm{Ax(\lambda) - b}_2^2} \max_{ e \in \lbrace e_1,\ldots,e_n \rbrace} \frac{|e' (A x(\lambda) - b) |^2}{\norm{ A'e }_2^2} \\
&= \quad \frac{1}{\lambda^2 \norm{Ax - b}_2^2}  \max_{ e \in \lbrace e_1,\ldots,e_n \rbrace} \frac{\lambda^2|e' (Ax- b) |^2}{\norm{ A'e }_2^2},
\end{aligned} 
\end{equation}
which implies that the threshold is magnitude invariant. Similarly, we can show that the selection probabilities are magnitude invariant (we look at the preceding calculation, but only for a nonempty subset of the equations).

\underline{Exploratory:} Let $V \subsetneq \rrow(A)$ be a nontrivial subspace. Then for any $x \in \mathbb{S}(x_{\row}^*) \cap V$, we saw that any equations for which $P_V A_{i,\cdot} = 0$ have a zero residual. Therefore, the only equations with nonzero residuals are those that not orthogonal to $V$. Since the threshold is bounded away from zero, only equations that are not orthogonal to $V$ can be in the subset. Therefore,
\begin{equation}
\sup_{V \subsetneq \rrow(A)} \sup_{ x \in \mathbb{S}(x_{\row}^*) \cap V } \Prb{ A'\varphi(A,b,x) \perp V } = 0.
\end{equation}
 
By the same argument as above, \cref{theorem:adaptive-convergence} follows. Similarly, $\E{ \tau_{\ell}} \leq \ell \dim( \rrow(A) )$. $\quad\blacksquare$

\paragraph{Sampling Kaczmarz-Motzkin.} In \cite{haddock2019} (Page 4), a subset of equations are randomly selected, and then the equation with the maximum residual is selected from this subset. 

\underline{$1$-Markovian:} The Sampling Kaczmarz-Motzkin method only relies on the current residual to sample. As a result, it is $1$-Markovian.

\underline{Magnitude Invariance:} The \textit{distribution} of the initial subsetting is independent and identical at each iteration. Therefore, conditioned on a given subset, we choose the maximum residual. By \cref{lemma:residual-mag-invar}, this last step is magnitude invariant. Moreover, since the random subsetting is independent and identical at each iteration, it too is magnitude invariant. Therefore, the entire procedure is magnitude invariant.

\underline{Exploratory:} Let $V \subsetneq \rrow(A)$ be a nontrivial subspace. Then, for any $x \in \mathbb{S}(x_{\row}^*) \cap V$, we have shown that there exists a $j$ such that $A_{j,\cdot}'x -b_j \neq 0$. Therefore, so long as the probability of selecting this equation is nonzero, then we are guaranteed that there is some choice of $\varphi(A,b,x)$ such that
\begin{equation}
\Prb{ A'\varphi(A,b,x) \perp V } \leq 1 - \Prb{ \mathrm{choosing ~ j ~ in ~ the ~ subset } }.
\end{equation}
Let $\pi$ be the smallest inclusion probability for any equation in the random subset. Then, it follows that 
\begin{equation}
\sup_{V \subsetneq \rrow(A)} \sup_{ x \in \mathbb{S}(x_{\row}^*) \cap V } \Prb{ A'\varphi(A,b,x) \perp V } \leq 1 - \pi.
\end{equation}
For the Sampling Kaczmarz-Motzkin method, the minimum inclusion probability is at least $\psi/n$, which corresponds to random sampling without replacement of subsets of size $\psi$. 

With these three properties verified and by \cref{lemma:convergence-rate-bound}, the conditions of \cref{theorem:adaptive-convergence} are satisfied and the result holds. The only thing left to show is that $\E{ \tau_{\ell}}$ are bounded by some $\ell \kappa$. By the proof of \cref{theorem:adaptive-convergence}, it is enough to bound the conditional expectations of $s_j$ in \cref{eqn:s_j}. Supposing  that $\pi$ for all $V \subsetneq \rrow(A)$, 
\begin{equation}
\condPrb{s_j = k}{\xi, s_1,\ldots,s_{j-1}, x_{\xi},\ldots,x_{\xi+s_1+\cdots+s_{j-1}+1 } } \leq (1 - \psi/n)^{k-1} \psi/n
\end{equation}
Hence, $\E{\tau_{\ell}} \leq \ell n\dim( \rrow(A))/\psi$. $\quad\blacksquare$

\section{Numerical Experiments} \label{section: experiments}
Here, we present a variety of numerical experiments to study the practicality of our approach in a sequential computing environment. Specifically, we test forty-nine systems with five hundred equations and five hundred unknowns. The coefficients are generated from forty-nine built-in matrices found in the \texttt{MatrixDepot} package for the \texttt{Julia} programming language \citep{zhang2016}. The solution to the equation is then generated using a standard, multi-variate normal vector. The constant vector is generated by the product of the two. Then, using the generated coefficient matrix and the generated constant vector, we solve the systems by varying the sample-generation method (i.e., the generation of $w$ and $\lbrace w_l \rbrace$) and the solver. The sample generation method is either produced by the count-sketch approach, the Gaussian approach, by uniformly sampling the equations of the matrix with replacement, or by uniformly sampling the equations of the matrix without replacement. The solver is  either a base method, the complete method, an intermediate method with $m=5$, or an intermediate method with $m=10$. Finally, we initialize $x_0 = 0$.

We recorded the wall clock time to improve the initial residual norm by a factor of ten with an upper bound of three seconds (note, a single iteration of a base method requires approximately $10^{-6}$ seconds, which allows the base method on the order of $10^6$ iterations on a $500 \times 500$ system). If the temporal upper bound is reached before a ten fold improvement in the initial residual norm is observed, the wall clock times is reported as $10^{99}$. Inherently, this metric results in a disadvantage for complete orthogonalization methods as such methods pay more for marginal improvements, but generate precise solutions with fewer iterations. However, with an eye towards solving much larger systems that require using a parallel or distributed environment, this metric of time-to-ten-fold improvement is the appropriate choice as the complete method would not be appropriate in such environments owing to the high communication costs that would be incurred. For the count-sketch sampling method, the wall clock times are reported in \cref{table:count_sketch}. For the remaining sampling approaches, the wall clock times are reported in the appendix.

\begin{table}[p]
\centering
\csvloop{
	tabular={lSSSS},
	file=outputs/MatrixMarket_CountSketch_10x-Improve-Time.csv,
	head to column names,
	before reading=\sisetup{table-format = 1.3e-1,table-number-alignment = center,table-auto-round},
	table head=\toprule\textbf{System} &\textbf{Base} &\textbf{Partial, $m=5$} &\textbf{Partial, $m=10$} &\textbf{Complete} \\\midrule,
	command=\Matrix&\Base&\PartFive&\PartTen&\Comp,
	table foot={\bottomrule}}
\caption{Wall clock time for ten-fold improvement of four solvers under the count-sketch sampling approach.}
\label{table:count_sketch}
\end{table}

While further analysis of each system would be necessary to understand the behavior of the solvers on each system, there are several important messages within \cref{table:count_sketch}. First, the base method often fails to achieve a ten-fold improvement despite the substantial number of iterations that the base solver is allowed (again, on the order of $10^6$). Unfortunately, the base method's poor behavior is observed even for the other sampling methods. Based on \cref{theorem:no-mem-convergence}, this would imply that either the stopping times $\lbrace \tau_l \rbrace$ are large and/or the rate of convergence (determined by $\lbrace \gamma_l \rbrace$) are too small. Given that this behavior is observed even for the random cyclic sampling case (which, by \cref{theorem-no-mem-w-o-replacement}, implies that the differences between the stopping times are bounded by a thousand), it is likely that the rate of convergence for such systems is close to unity. 

However, we see a tremendous benefit even from a small amount of partial orthogonalization. That is, the intermediate solvers with $m=5$ and $m=10$ perform quite well. In particular, whenever complete orthogonalization achieves a ten-fold improvement within the allotted time, the partial orthogonalization methods also achieve the ten-fold improvement within the allotted time and often orders of magnitude faster. Thus, for cases when the base method performs poorly, a small amount of partial orthogonalization is able to usually able to remedy this poor behavior. One final observation is that the $m=5$ method often outperforms the $m=10$ method. This seems to be because of the memory-management and garbage collection time related to modifying the set $\mathcal{S}$, which we did not optimize for these experiments. Thus, a more complete study would require a detailed optimization of how $\mathcal{S}$ is handled.

\section{Conclusion} \label{section:conclusion}
To reiterate, our motivation was to address the two practical challenges of the typical sketch-\textit{then}-solve approach for solving linear systems. These practical challenges are: there is no clear way of choosing the size of the sketching matrix \textit{apriori}; and there is a nontrivial storage costs of the sketched system. We made progress towards addressing these challenges by reformulating the sketch-\textit{then}-solve approach to a sketch-\textit{and}-solve approach in which the sketched system is implicitly constructed and solved simultaneously. The main idea of the reformulation is to construct the equations of the sketched system one-at-a-time and then use an orthogonalization scheme to solve the system as each sketched equation is observed. As a result, we addressed the concern of determining the sketching matrix's dimensions because, under our reformulation, the sketching matrix could be grown to an arbitrary size during the calculation up to a user-defined stopping criteria, which may be based on closeness to a solution or based on a computational budget. Moreover, we addressed the cost of storing the sketched system because we do not need to explicitly form the entire sketched system under our reformulation. However, we traded this storage problem with another one---albeit less onerous---of storing the matrix $S$. Finally, we address the overlooked practical challenge of solving the sketched system by using our orthogonalization scheme to solve the implicitly sketched system under our reformulation. 

When $d$ becomes very large, storing and manipulating $S$ becomes prohibitive. Because of the challenges introduced by $S$, we proposed intermediate methods that implicitly stored $S$ using only a handful of vectors. The result was a collection of partial orthogonalization schemes, and, in the limit of not storing any vectors for $S$ (i.e., $S$ becomes the identity), we recovered what we called ``base methods,'' which included the important special cases of randomized Kaczmarz and randomized Gauss-Seidel. As a result, we were able to make a conceptual connection between random sketching method (i.e., complete orthogonalization methods under our formulation) and the usual randomized iterative methods (i.e., base methods under our formulation). Importantly, we were able to leverage this conceptual relationship between the two to connect the convergence theory of the complete orthogonalization method to the convergence theory of the base methods. The key ingredient here is that the stopping time that was defined for the complete orthogonalization method encoded information about exploration of a subspace that contained the solution of the sketched system. This stopping time was then used (in a repeated fashion) to guarantee that a certain amount of progress for the base methods is achieved. As a result, we were able to produce a convergence theory for these base methods that was both quite general and complemented and improved on previous results. In fact, we were able to use this theory to prove convergence for a broad class of adaptive sampling methods, and supply rates of convergence.

The predominant missing component of this work is the rigorous analysis of the so-called intermediate methods that reside between the base methods and the complete methods. Such an analysis is certainly warranted owing to the impressive numerical performance of these intermediate methods as demonstrated in our experiments. Owing, primarily, to the additional complexity of analyzing such intermediate methods and, secondarily, of space limitations, a rigorous analysis of these methods will be the focus of future work. Additionally, an efficient implementation at scale for challenging problems arising in partial differential equations with a detailed comparison to existing state-of-the-art methods will be included in future work.

\small
\bibliographystyle{plainnat}
\bibliography{bibliography}

\begin{appendix}
\section{Numerical Experiments, Continued}
Here, we present the wall clock time to a ten-fold improvement for the remaining sampling schemes: Gaussian (\cref{table:gaussian}), uniformly choosing an equation with replacement (\cref{table:row-wr}), uniformly choosing an equation without replacement (\cref{table:row-cyc}). 

Again, we recorded the wall clock time to improve the initial residual norm by a factor of ten with an upper bound of three seconds (note, a single iteration of a base method requires approximately $10^{-6}$ seconds, which allows the base method on the order of $10^6$ iterations on a $500 \times 500$ system). If the temporal upper bound is reached before a ten fold improvement in the initial residual norm is observed, the wall clock times is reported as $10^{99}$. Inherently, this metric results in a disadvantage for complete orthogonalization methods as such methods pay more for marginal improvements, but generate precise solutions with fewer iterations. However, with an eye towards solving much larger systems that require using a parallel or distributed environment, this metric of time-to-ten-fold improvement is the appropriate choice as the complete method would not be appropriate in such environments owing to the high communication costs that would be incurred.

\begin{table}[p]
\centering
\csvloop{
	tabular={lSSSS},
	file=outputs/MatrixMarket_Gaussian_10x-Improve-Time.csv,
	head to column names,
	before reading=\sisetup{table-format = 1.3e-1,table-number-alignment = center,table-auto-round},
	table head=\toprule\textbf{System} &\textbf{Base} &\textbf{Partial, $m=5$} &\textbf{Partial, $m=10$} &\textbf{Complete} \\\midrule,
	command=\Matrix&\Base&\PartFive&\PartTen&\Comp,
	table foot={\bottomrule}}
\caption{Wall clock time for ten-fold improvement of four solvers under the Gaussian sampling approach.}
\label{table:gaussian}
\end{table}

\begin{table}[p]
\centering
\csvloop{
	tabular={lSSSS},
	file=outputs/MatrixMarket_KaczmarzWR_10x-Improve-Time.csv,
	head to column names,
	before reading=\sisetup{table-format = 1.3e-1,table-number-alignment = center,table-auto-round},
	table head=\toprule\textbf{System} &\textbf{Base} &\textbf{Partial, $m=5$} &\textbf{Partial, $m=10$} &\textbf{Complete} \\\midrule,
	command=\Matrix&\Base&\PartFive&\PartTen&\Comp,
	table foot={\bottomrule}}
\caption{Wall clock time for ten-fold improvement of four solvers when equations are sampled uniformly with replacement.}
\label{table:row-wr}
\end{table}

\begin{table}[p]
\centering
\csvloop{
	tabular={lSSSS},
	file=outputs/MatrixMarket_KaczmarzCyc_10x-Improve-Time.csv,
	head to column names,
	before reading=\sisetup{table-format = 1.3e-1,table-number-alignment = center,table-auto-round},
	table head=\toprule\textbf{System} &\textbf{Base} &\textbf{Partial, $m=5$} &\textbf{Partial, $m=10$} &\textbf{Complete} \\\midrule,
	command=\Matrix&\Base&\PartFive&\PartTen&\Comp,
	table foot={\bottomrule}}
\caption{Wall clock time for ten-fold improvement of four solvers when equations are sampled uniformly without replacement.}
\label{table:row-cyc}
\end{table}
\end{appendix}

\end{document}